\documentclass[11pt,a4paper,reqno]{amsart}

\pdfoutput=1

\usepackage{amsmath,amsfonts,amssymb,amsthm,color}  
\usepackage{mathrsfs}
\usepackage[utf8]{inputenc}
\usepackage{graphicx}
\usepackage{hyperref}

\theoremstyle{plain}
\newtheorem{theorem}{Theorem}[section]
\newtheorem{corollary}[theorem]{Corollary}
\newtheorem{lemma}[theorem]{Lemma}
\newtheorem{proposition}[theorem]{Proposition}

\theoremstyle{definition}
\newtheorem{definition}[theorem]{Definition}

\newtheorem{remark}[theorem]{Remark}

\numberwithin{equation}{section}
\newtheorem*{theorem*}{Theorem}

\newcommand{\R}{{\mathbb R}}
\newcommand{\rn}{{\mathbb R}^n}
\newcommand{\N}{{\mathbb N}}

\newcommand{\Tcal}{\mathcal{T}}

\DeclareMathOperator{\III}{III}
\DeclareMathOperator{\II}{II}
\DeclareMathOperator{\I}{I}
\DeclareMathOperator{\dist}{dist}
\DeclareMathOperator{\diam}{diam}
\DeclareMathOperator{\dive}{div}
\DeclareMathOperator{\supp}{supp}

\providecommand{\ab}[1]{  \lvert  #1  \rvert }
\providecommand{\abs}[1]{ \left \lvert  #1 \right \rvert }

\providecommand{\no}[1]{  \lVert  #1  \rVert }

\providecommand{\la}[1]{  \langle #1 \rangle}

\providecommand{\ddd}{\, dx dt }

\providecommand{\trm}[1]{\textrm{#1}}

\providecommand{\seb}[1]{\textcolor{blue}{#1}}

\title[Right inverse of the divergence]{Construction of a right inverse for the divergence in non-cylindrical time dependent domains}
\author{Olli Saari}
\author{Sebastian Schwarzacher}

\address{Mathematical Institute, 
	University of Bonn,
	Endenicher Allee 60, 53115, Bonn,
	Germany}
	\email{saari@math.uni-bonn.de}
	
\address{Katedra matematick\'e analy\'zy, Matematicko-fyzik\'aln\'\i\
  fakulta Univerzity Karlovy, Sokolovsk\'a 83, 186 75 Praha 8, Czech Republic}
  \email{schwarz@karlin.mff.cuni.cz}

  \makeatletter
  \@namedef{subjclassname@1991}{{\rm 2020} Mathematics Subject Classification}
  \makeatother


\usepackage[normalem]{ulem}
\begin{document}

\begin{abstract}
  We construct a stable right inverse for the divergence operator
  in non-cylindrical domains in space-time. 
  The domains are assumed to be H\"older regular in space and evolve continuously in time.
  The inverse operator is of Bogovskij type, 
  meaning that it attains zero boundary values. 
  We provide estimates in Sobolev spaces of positive and negative order 
  with respect to both time and space variables.
  The regularity estimates on the operator depend 
  on the assumed H\"older regularity of the domain. The results can naturally be connected to the known theory for Lipschitz domains. The most precise estimates are given in weighted spaces, where the weight depends on the distance to the boundary. This allows for the deficit to be captured precisely in the vicinity of irregularities of the boundary.
  As an application, 
  we prove refined pressure estimates for weak and very weak solutions to Navier--Stokes equations in time dependent domains.
\end{abstract}

\maketitle

\section{Introduction}
Consider an open and connected subset $ \Omega \subset \R^{1+n} $ of the space-time.
We denote by 
\[\Omega_t:=\{t\}\times \{x:(t,x)\in \Omega\}\] 
the time slice of $\Omega$ at time $t\in \R$.
Also the time slices are assumed to be connected.
We study the problem of constructing vector fields 
\[u=(u_1,...,u_n):\Omega\to \R^n\] 
that satisfy
\begin{align*}
%\label{eq:1}
\begin{aligned}
\dive u(t,x) := \sum_{i=1}^{n}\partial_{x_{i}} u_i(t,x_1, \ldots , x_n) &= f(t,x), \quad (t,x) \in \Omega,
\\
u(t,x)&=0, \quad (t,x) \in \bigcup_{t \in \R}\partial\Omega_t,
\end{aligned}
\end{align*}
where the functionals $ f(t, \cdot) $ are subject to $\la{f(t,\cdot),1} = 0$.
For brevity, 
we call the spatial divergence operator just divergence in what follows.

The existence and properties of the right inverse %of the divergence 
are directly linked to applications in continuum mechanics. 
Most prominent are the applications in fluid mechanics, 
see for example~\cite{MR2808162,FeiNovPel01} and Subsection~\ref{sec:introfluid}.
Domains that vary in time are certainly meaningful
from the physical point of view,
and not surprisingly, 
an increasing body of literature has been devoted to
the mathematical theory of fluids in this setting,
see for instance \cite{canic2020moving} and the references therein.  
Our construction of a stable-in-time right inverse %of the divergence 
extends these results. 
Below, we use the here constructed right inverse in order to  construct pressure terms in natural associate spaces, provided a velocity field obeying the (very) weak incompressible Navier--Stokes equations
has been given.

Although the divergence equation has been studied quite extensively 
in domains independent of time, see \cite{MR3618122},
the treatise of time-dependent domains seems to be completely missing
in the literature. 
In addition, 
it is difficult to find a complete treatment of H\"older regular domains in the literature,
even in the steady setting.
Namely,
first order Sobolev estimates in planar domains are treated in \cite{MR2198138},
extremely general results in \cite{MR2674865} certainly cover the H\"older setting,
and even sharp results are likely to follow from the treatise in \cite{MR3273166},
but it is still behind a modicum of work to extract the desired estimates from these references.
The objective of the present paper is hence twofold. 
We treat a class of non-cylindrical settings in space-time as our main task,
but as a side product, 
we also recover results for H\"older regular domains in the steady setting.
%, which is contained in our results.

As the divergence is a pure space operator,
one might expect that its right inverse would commute with differentiation in time.
However,
if one asks the solution to satisfy a boundary condition,
the right inverse starts depending on time through the evolution of the boundary,
and the question about time regularity becomes non-trivial.
The dependency of the right inverse on the geometry of the domain
is typically very in-explicit
in the stationary setting,
and solutions regular in time cannot be produced by
a straight-forward repetition of the constructions slice by slice.
Such a slicing argument is enough for some of the usual applications,
such as slice-wise versions of Korn's inequality and Lions--Ne\v{c}as theorem on negative norms,
but it is not sufficient for the application we present at the end of the introduction.

We propose a space-time approach
and construct solutions that have some regularity in the time variable
and satisfy the same space regularity and the same boundary conditions
as solutions constructed for a single time slice are expected to satisfy.
% We treat bounded domains whose boundary is locally a graph of a function
% H\"older continuous in space and uniformly continuous in time,
% which is the class of domains we consider the most important from the point of view of applications.
We demonstrate how improving the regularity of the boundary 
reflects as improved regularity of the solution operator. The most precise estimates are given in weighted spaces, where the weight depends on the distance to the boundary. The approach seems to appear naturally once singular boundaries are allowed (see \cite{MR2674865}). Indeed, it allows for the deficit to be captured precisely in the vicinity of irregularities of the boundary, hence underlying the {\em locality} of the operator; or in other words revealing the strength of the construction for losing regularity merely close to singularities.
Our construction combines insights from several existing solution strategies in the stationary setting \cite{MR2643399,MR2674865,MR631691} as well as a novel component to deal with the time variable,
our main object of interest.

Before stating our main results,
we clarify the background and
briefly review the theory on solving the divergence in the stationary set-up.
The divergence equation in its usual applications is ill-posed,
and solutions are never expected to be unique,
not even under a boundary condition.
The solution given by the Bogovskij--Sobolev (briefly Bogovskij in what follows) integral formula \cite{MR631691}
\begin{equation*}
  B_{B(0,1)} f (x) := \int_{\rn} f(y)(x-y) \int_{1}^{\infty} \delta^{-n} b\left( y + r(x-y)  \right) r^{n-1} \, d r d y
\end{equation*} 
with $ f $ a test function, and $ b : \rn \to [0,1] $ a smooth bump function of total mass one and support in $B(0,1)$, has the remarkable property
of preserving the class of test functions compactly supported in a domain star-shaped 
with respect to $ B(0,1) $.
By Calder\'on--Zygmund theory, it maps $ W^{s,p} $ to $ W^{s+1,p} $ 
for all $ 1<p<\infty $ and all $ s \in \R $,
and hence Bogovskij's formula gives solutions whose boundary values are zero
in the Sobolev sense.
A by-now-standard argument implies existence of a solution operator with these properties in arbitrary Lipschitz domains \cite{MR0227584}.
The line integrals in the Bogovskij formula can be replaced by more general curves,
and based on this modification, 
the Bogovskij formula can be generalized to bounded John domains \cite{MR2263708}.

Every John domain, 
or more generally every domain satisfying an emanating chain condition,
admits a decomposition operator mapping $ L^{p} $ functions with mean zero 
boundedly to $ \ell^{p} $ sequences of compactly and disjointly supported $ L^{p} $ functions,
again with mean zero. 
Such a decomposition operator has been constructed in \cite{MR2643399}, 
and it can be applied 
to prove a weighted Poincar\'e inequality 
and to solve the divergence equation on $ A_p $ weighted Sobolev spaces,
thus providing an alternative approach to solving the divergence.
The formula of the decomposition operator resembles 
a discrete Bogovskij formula,
and it is indeed shown in \cite{MR2674865} 
that solvability of the divergence in certain weighted $ L^{p} $ spaces 
conversely implies the existence of a decomposition operator as in \cite{MR2643399}.
The divergence equation in weighted $ L^{p} $, in turn,
can be solved by means of an abstract duality argument based on an improved Poincar\'e inequality,
and to close the circle,
the result in \cite{MR3178379} shows
that solvability of the divergence equation,
an improved Poincar\'e inequality
and the John condition 
are equivalent for simply connected and bounded planar domains.
A Korn type inequality is added to the list of equivalences in \cite{MR3669778}.
Other results in the stationary setting include extensions to Besov and Triebel-Lizorkin spaces~\cite{CosMcI10}, measure data problems \cite{moonens2020solvability} and boundary value problems in non-smooth domains \cite{MR2425010},
to mention a few.

All the constructions involved in the results above are somewhat in-explicit in their dependency on the geometry of the domain,
but the last example indicates that the geometry plays a prominent role in the construction.
Bogovskij's original formula, which is domain independent, is an exception.
Applying either that or any of the more complicated constructions 
to functions of the space-time by plainly ignoring the time variable,
we see that domains either independent of time
or staying star-shaped with respect to a fixed ball 
can be treated easily,
the stationary results holding on all time slices and 
the construction commuting with the time derivatives.
All other domains seem to be beyond reach for the method of trivial extension in time,
and this is the starting point of our work.
% The constraints just mentioned exclude many important domains, 
% such as evolutions creating and smoothing away cusps and corners.
% ,
% many of which are covered by the conditions on domains we discuss next.

We denote by $ C^{\alpha,\beta,\theta} $ the class of domains enclosed by graphs of functions 
that are $ \beta$-H\"older continuous in space, 
$\alpha$-H\"older continuous in time
and whose boundaries are thin according to parameter $\theta$, 
see Definition \ref{def:domains}.
For every bounded H\"older domain there is an admissible parameter $\theta \in [\beta,1]$.
The case $ \alpha, \beta , \theta = 0 $ is included as the general bounded graph domains.
The graph domains can allow exterior cusps unlike John domains,
but typical fractal boundaries such as the von Koch snowflake are excluded.
Fractal boundaries are possible for John domains,
and hence the class of graph domains is incomparable with the John condition.
A more general $ s $-John condition would unify both conditions,
but its treatment is beyond our reach in the time-dependent setting.
We refer to Subsection \ref{sec:domain} for a more thorough discussion of the assumptions on domains.

For the rest of the introduction,
we let $\theta \in [0,1]$ and fix a domain $ \Omega $ in $ C^{0,0,\theta} $.
We assume $\Omega \subset (0,T) \times \R ^{n}$ for some $T > 0$
and specify the meaning of the number $\theta$ as 
\begin{equation}
  \label{eq:defthetaintro}
\sup_{t} | \{ x \in \Omega_t : \dist(x,\partial \Omega_t) \le \varepsilon  \}| \le C \varepsilon^{\theta} 
\end{equation}
holding for all $\varepsilon > 0$.
In general, $\theta $ can be zero.
If all $\Omega_t$ are $\beta$-H\"older,
then $\theta \ge \beta $,
and in the case of a Lipschitz domain, 
whose regularity is possibly broken by finitely many power type cusps,
we can take $\theta = 1$.
% See the discussion related to the auxiliary Proposition \ref{prop:tounweighted}
% for the role of $\theta$ in the estimates.
In Theorem \ref{thm:bogovski_on_w1p}, 
we construct a linear operator $ B $
acting on test functions in (slice mean zero class)
\[
 C_{smz}^{\infty}(\Omega) :=  \left \lbrace f \in C^{\infty}(\Omega) \cap C(\overline{\Omega}):  \sup_{t} \left \lvert\int f(t,x) \, dx \right  \rvert  = 0 \right \rbrace
\]
such that for $ f \in C_{smz}^{\infty}(\Omega) $ with compact support
\[
  \dive B f = f, \quad B f \in C_0 ^{\infty}(\Omega).
\]
The focus of our work is on local properties of Bogovskij operators close to the boundary. The large-scale geometric set-up is kept rather general as are the dependencies on them in the various norm estimates provided in this work. Therefore (in order not to distract from the qualitative properties of the here introduced operator), we keep the dependencies of the bounds on the geometry in the statements in the introduction implicit and collect them in Remark \ref{remarks-after-proof}. We believe that there is room of improvement for sharp and explicit dependencies on the large-scale geometric parameters of domains, if further restriction are made.

The stability and regularity estimates for the operator
are either weighted estimates or corollaries of weighted estimates.
The weight is always the distance to the spatial boundary of the domain with some power. 
Such dependencies are natural in two ways.
First, they show that the smoothness in the interior is always satisfied.
Second, they show that the zero trace of our inverse is attained in a sense weaker than usual
if the boundary of the domain is less regular than Lipschitz in space.

The best-behaved special case of spatially Lipschitz regular domains
admits estimates slightly weaker 
than what is known in the case of cylindrical domains here.
The time derivative is effected by the evolution of the boundary 
as indicated by the weight function appearing on the right hand side of the estimate.
The weight function does not appear when the domain is uniformly star-shaped
in the sense that each time slice is star-shaped with respect to a fixed ball.
The Bogovskij operator defined with respect to this fixed ball then acts as a right inverse of the divergence in the whole time dependent domain.
However, we believe that when leaving this regime 
(in particular when going beyond Lipschitz regularity
as we do in this paper) a defect with a weight function is unavoidable.

\begin{theorem}
  Let $0 < \alpha \le 1$.
  Assume $ \Omega $ to be a $ C^{\alpha,1,\theta} $ domain,
  $ 1<p< \infty$ and
  $ \kappa, k \ge 0 $ be integers. 
  Then for all compactly supported test functions $f \in C_{smz}^{\infty}(\Omega)$ and all times $t$
  \begin{multline*}
  \no{\partial_t^{\kappa}Bf(t,\cdot)}_{\dot{W}^{k+1,p}(\Omega_t)} \\
  \le C \left( \sum_{\lambda=0}^{\kappa} \sum_{ |\gamma |=k} \int_{\Omega_t} |\partial_t^{\lambda}\partial^{\gamma} f(t,x) |^{p} \dist(x,\partial \Omega_t)^{p(\lambda- \kappa)/\alpha} \, dx \right)^{1/p}
  \end{multline*}
  with $ C $ only depending on $ \Omega $, $ k $, $ \kappa $ and $ p $.
\end{theorem}

The first case of our main results to be highlighted is that of first order bounds 
in general $ C^{\alpha,\beta,\theta} $ domain.
\begin{theorem}
  Assume that $ \Omega $ is $ C^{\alpha,\beta,\theta} $ with $\alpha, \beta, \theta \in (0,1] $,
  $ 1<p,q<\infty $ and 
  $ \kappa \ge 0 $ an integer.
  Then the a priori bound
  \begin{multline*}
  \no{\partial_t^{\kappa} \nabla Bf(t, \cdot )}_{L^{q}(\Omega_t)} \\
  \le 
  C \left( \sum_{\lambda=0}^{\kappa} \int_{\Omega_t} |\partial_t^{\lambda} f(t,x) |^{p} \dist(x,\partial \Omega_t)^{p(\lambda- \kappa)/(\alpha \beta)} \, dx \right)^{1/p} 
  \end{multline*}
  holds for all test functions $f \in C_{smz}^{\infty}(\Omega)$, all $t$ and all $q > 0$ such that  
  \[
  \frac{(1+\varepsilon)(1-\beta)}{\theta} \le \frac{1}{q} - \frac{1}{p}
  \]
  for some $\varepsilon > 0$.
  The constant $ C $ only depends on $ \Omega $, $ \kappa $, $ p $ and $ q $.
\end{theorem}
As the parameter $\theta$ is bounded from below by $\beta$,
we see that the estimate improves as $\beta$ gets larger.
In addition, if $\partial \Omega_t$ happens to be rectifiable with finite Hausdorff measure,
we get an estimate with $\theta = 1$, compare to Corollary \ref{cor:rectifiable}.

As the second special case,
we state the estimate without time derivatives,
which reproduces the stationary estimates.
\begin{theorem}
\label{thmintro:notime}
  Assume that $ \Omega $ is $ C^{0,\beta,\theta} $ with $ \beta \in (0,1] $,
  $ 1<p,q<\infty $ and 
  $ k  \ge 0 $ an integer.
  Then the a priori bound
  \begin{multline*}
  \no{  Bf(t, \cdot )}_{\dot{W}^{k+1,q}(\Omega_t)} \\
  \le 
  C \left( \sum_{ |\gamma| \le k } \int_{\Omega_t} |\partial^{\gamma} f(t,x) |^{p} \dist(x,\partial \Omega_t)^{p(|\gamma| - k)/\beta} \, dx \right)^{1/p} 
  \end{multline*}
  holds for all test functions $f \in C_{smz}^{\infty}(\Omega)$, all $t$ and all $q > 0$ such that  
  \[
  \frac{(1+\varepsilon)(1-\beta)}{\theta} \le \frac{1}{q} - \frac{1}{p}
  \]
  for some $\varepsilon > 0$.
  The constant $ C $ only depends on $ \Omega $, $ k $, $ p $ and $ q $.
\end{theorem}

Under an additional assumption on the domain that Hardy's inequality hold, 
see \eqref{eq:hardy_lipschitz},
the estimate above can be simplified further.
An important example of such domains are the simply connected and bounded planar $\beta$-H\"older domains $\Omega_t$.
\begin{theorem}
Assume the notation of Theorem \ref{thmintro:notime}
and assume in addition that $\Omega_t$ satisfies Hardy's inequality~\eqref{eq:hardy_lipschitz} for $b\leq 0$.
Under the same conditions on $p$ and $q$ as above,
there is a constant $C$ such that for all compactly supported $f \in C_{smz}^{\infty}(\Omega)$
\[
    \no{  Bf(t, \cdot )}_{\dot{W}^{k+1,q}(\Omega_t)} \le C \no{  f(t, \cdot )}_{\dot{W}^{k,p}(\Omega_t)} .
\]
\end{theorem}

The best estimates that follow from our construction
in the general case are recorded 
in Theorem \ref{thm:high}.
The general form comes with weights quantifying scaling deficits 
relative to both spatial and temporal irregularity of the boundary.
The reason why estimates in Lipschitz domains are simpler is two-fold.
On one hand, 
one can set $\beta = 1$ in the estimate
but on the other hand one may also simplify the expression by means of Hardy's inequality. 
Concerning the space-time setting, 
the difference between Lipschitz in space and H\"older in space domains is even more drastic. 
For domains that are uniformly Lipschitz in space, 
in the sense that they are uniformly star-shaped with respect to a fixed set of balls,
there exist constant in time Bogovskij operators, 
but
operators of such characteristics are highly unlikely to exist
in a space-time domain that allows for an exterior cusp that moves as time passes. 

Finally we give proposition 
that demonstrates that our construction is flexible enough 
for unweighted estimates useful for applications.
The proposition below is proved after Proposition \ref{prop:new}.
More results of similar flavor can be deduced from Theorem \ref{thm:negat}.  
\begin{proposition}
\label{prop_last_intro}
Assume that $\Omega\in C^{\alpha,\beta,\theta}$.
Fix $1< p,s < \infty$, a time $t$, and assume $f$ is a limit of $C_{smz}^{\infty}(\Omega)$ functions in $L^{1}(\Omega)$ norm 
with 
$\partial_t f \in W_0^{-1,p}(\Omega_t)$ and 
$f\in L^{s}(\Omega_t)$.
Then
$\partial_t B f \in L^{q}(\Omega_t)$
for all $q > 0$ such that 
\[
\frac{(1+ \varepsilon)(1-\beta)}{\theta} \le   \frac{1}{q} - \frac{1}{p} \quad \text{and} \quad \frac{(1+ \varepsilon)(1/\alpha -\beta)}{\theta} \le   \frac{1}{q} - \frac{1}{s}
\]
for some $\varepsilon > 0$.
\end{proposition}
In particular, 
if $\partial_t f\in W_0^{-1,2}(\Omega_t)$, $f\in L^s(\Omega_t)$, 
$\theta=1$  
and 
\[
q<\min \left \lbrace \frac{\alpha s}{s + \alpha - \alpha s \beta},\frac{2}{3-2\beta} \right \rbrace,
\]
then $\partial_t B f\in L^{q}(\Omega_t)$.

Generally it seems noteworthy that the regularity loss is due to the fact that the zero boundary values are attained in a weaker way in non-Lipschitz domains. Or in case of Lebesgue spaces (or negative order Sobolev spaces), more severe singularities are admissible, which are revealed by the integration of the Bogovskij operator.

\subsection{Application to the Navier Stokes equation on non-cylindrical domains}
\label{sec:introfluid}
An incompressible fluid is described by a velocity field $ v: \Omega \to \rn $
and a pressure $ \pi : \Omega \to \R  $
that are subject to the Navier--Stokes equations 
\begin{align} 
  \begin{aligned}
    \label{eq:NS}
    \rho \dfrac{D}{Dt}v(t,x)-\mu\Delta v(t,x)+\nabla \pi(t,x)
    &=\rho g(t,x) +\dive F(t,x) \\
    \dive v(t,x) &=0.
    \end{aligned}
\end{align}
Here 
$\rho$ is the density of the fluid,
\[
  \dfrac{D}{Dt}v(t,x) =\partial_t v(t,x)+[\nabla v(t,x)]v(t,x) 
\]
is the material derivative, $\mu>0$ is the viscosity of the fluid, 
and $g: \Omega \to \rn$ and $F:\Omega \to  \rn$ 
describe external forces such as gravity. 
We assume $g,F \in L_{loc}^{1}(\Omega)$.
Denote by
\[
C_{sol,0}^{\infty}(\Omega;\rn) = \{ \varphi \in C_0^{\infty}(\Omega;\rn): \dive \varphi = 0 \}
\]
the class of solenoidal and compactly supported test functions.

\begin{definition}
\label{def:fluid_intro}
A very weak solution in $\Omega \in C^{0,0,\theta}$ is a pair $(v,\pi)$
where
\begin{itemize}
  \item the velocity field $v \in L_{loc}^{2}(\Omega;\rn)$ satisfies 
  \begin{align*}
    -\la{ v,\partial_t \psi} - \la{v \otimes v ,\nabla \psi}- \mu \la{ v , \Delta \psi }
    -  \la{g ,\psi} + \la{F ,\nabla \psi} &= 0 \\
    \la{v, \nabla \psi} &= 0
  \end{align*}
  for all test functions $\psi \in C_{sol,0}^{\infty}(\Omega;\rn)$;
  \item the pressure $\pi $ is an element of $ (C^{\infty}_0(\Omega))^{*}$ and $-\nabla \pi$ equals the distribution mapping a general, not necessarily solenoidal, test function $\psi$ to the left hand side of the first equation above.
\end{itemize}
\end{definition}
The existence of weak solutions is known under very general conditions. 
Up to the authors' knowledge, the most general results can be found in \cite{neustupa09}.
Remarkably, the pressure is kept in a rather implicit form in \cite{neustupa09} 
as well as in many other existence results. 
Rather often it does not even appear in the definition of weak solutions, even in the case when the domain is assumed 
to have a Lipschitz regular boundary for all times \cite{NeuPan09,LenRuz14,MR1948464}.
Starting from such a fluid velocity field without explicit pressure,
the right inverse of the divergence introduced in the present paper 
can be used as an efficient tool to construct missing pressures 
with appropriate regularity properties.

We state a sample result in general H\"older domains
of the space-time. 
\begin{theorem}
  \label{thmint:fluid3}
Let $\Omega$ be a $C^{\alpha,\beta,\theta}$ domain with $\alpha, \beta \in (0,1]$.
Let $2 < p, q < \infty$ and $1< r,s < \infty$,
and assume that $v$, $F$ and $g$ are a velocity field and two force terms as in Definition \ref{def:fluid_intro}
such that 
\[
v \in L^{q}L^{p}(\Omega), \quad g, |F| \in L^{s}L^{r}(\Omega).
\]
Then there exists a pressure decomposition 
\[
\pi = \pi_{time}^1+\pi_{time}^2 + \pi_{conv} + \pi_{visc} + \pi_{ext,1} + \pi_{ext,2}
\] 
such that $(v,\pi)$ is a very weak solution in $\Omega$
and 
\begin{align*}
  &\pi_{time}^1 \in W^{-1,q}W^{1,p_1}(\Omega), \quad
  \pi_{conv} \in L^{q}L^{p_1/2}(\Omega), \quad
  \pi_{ext,2} \in L^{s}L^{r_1}(\Omega),
  \\
  &\pi_{time}^2 \in L^{q}L^{p_2}(\Omega), \quad 
  \pi_{visc} \in L^{q}W_0^{-1,p_1,-\beta,0,\infty,0}(\Omega),\quad
  \pi_{ext,1} \in L^{s}W^{1,r_1}(\Omega),
\end{align*}
whenever $1< p_1, p_2, r_1 < \infty $ satisfy 
\begin{align*}
\frac{(1+ \varepsilon) (1- \beta) }{\theta}  &\le \frac{1}{p_1} - \frac{1}{p}, \\
\frac{1+ \varepsilon }{\theta} \left(\frac{1}{\alpha} -  \beta \right) &\le \frac{1}{p_2} - \frac{1}{p} ,\\
\frac{ (1+ \varepsilon) (1- \beta) }{\theta} &\le \frac{1}{r_1} - \frac{1}{r}
\end{align*}
for some $\varepsilon > 0$.
\end{theorem}
As the boundary of the domain is merely $\beta$-H\"older,
the lack of boundary regularity has an effect on the pressure, 
even when the fluid velocity is a weak solution. 
The estimates above are for very weak solutions. 
In the case of weak Leray solutions, 
one can prove better estimates.
All the estimates above suffer from the lack of regularity in space,
but only the pressure term related to the time derivative of the velocity
produces a contribution stemming from the irregularity of the evolution of the domain. 
On the other hand, if the domain happens to be Lipschitz in space,
we can simplify a bit further.
\begin{theorem}
\label{thmint:fluid4}
Assume the domain in Theorem \ref{thmint:fluid3} is $C^{\alpha,1,1}$,
that is, Lipschitz in space.
Then for 
\[
v \in L^{q}L^{p}(\Omega), \quad g, |F| \in L^{s}L^{r}(\Omega),
\]
we get 
\begin{align*}
  &\pi_{time}^1 \in W^{-1,q}W^{1,p}(\Omega), \quad
  \pi_{conv} \in L^{q}L^{p/2}(\Omega), \quad
  \pi_{ext,2} \in L^{s}L^{r}(\Omega),
  \\
  &\pi_{time}^2 \in L^{q}L^{p_2}(\Omega), \quad 
  \pi_{visc} \in L^{q}W^{-1,p}(\Omega),\quad
  \pi_{ext,1} \in L^{s}W^{1,r}(\Omega),
\end{align*}
whenever $1< p_2 < \infty$ satisfies
\begin{align*} 
 \left( \frac{1}{\alpha} - 1 \right) < \frac{1}{p_2} - \frac{1}{p}  .
\end{align*}
\end{theorem}
This estimate is included for comparison to show how the weighted estimates are connected to unweighted spaces.
%However, a more specific construction would allow to keep $\pi_{visc}$ in the negative Sobolev space with zero boundary values while our construction needs to relax the boundary condition here.

We also give another formulation for the construction of pressures in terms of weighted Sobolev spaces. 
% The weight contributes to the norms significantly close to the boundary of the domain 
% and it indicates that the zero boundary values are attained at the singularities of the boundary in a form considerably weaker 
% than what is the case for smooth boundary points. 
% Nevertheless,
% the weighted Sobolev spaces we consider are embedded in unweighted Sobolev spaces 
% with less gradient integrability,
% which allows one to write down all the results in several forms.
% The previous theorem is an example to that direction.
What follows, 
is (probably) a more precise result using a less transparent formulation with weights.
See Section \ref{sec:function} for the notation.
\begin{theorem}
  \label{thmint:fluid2}
Let $\Omega$ be a $C^{\alpha,\beta,\theta}$ domain with $\alpha, \beta \in (0,1]$.
Let $2 < p, q < \infty$ and $1< r,s < \infty$,
and assume that $v$, $F$ and $g$ are a velocity field and two force terms as in Definition \ref{def:fluid_intro}
such that 
\begin{align*}
  &\int \left(\int |v(t,x)|^{p} \dist (x, \partial \Omega_t) ^{(\beta - 1)p} \, dx \right)^{q/p} \, dt < \infty \\
  & \int \left( \int (|F(t,x)|^{r}  + |g(t,x)|^{r})\dist (x, \partial \Omega_t) ^{(\beta - 1)r} \, dx \right)^{s/r}  dt < \infty.
\end{align*}
Then there exists a pressure decomposition 
\[
\pi = \pi_{time} + \pi_{conv} + \pi_{visc} + \pi_{ext,1} + \pi_{ext,2}
\] 
such that $(v,\pi)$ is a very weak solution in $\Omega$
and 
\begin{multline*}
  \pi_{time} \in L^{q}W^{1,p,\infty,-1,-\alpha \beta,0}, \quad 
  \pi_{conv} \in L^{q} L^{p/2}(\Omega), \quad 
  \pi_{ext,2} \in L^{s}L^{r}(\Omega), 
  \\
  \pi_{visc} \in L^{q}W_0^{-1,p,-\beta,0,\infty,0}(\Omega),\quad
  \pi_{ext,1} \in L^{s}W^{1,r}(\Omega).
\end{multline*}
\end{theorem}

We refer to Theorem \ref{thm:fluid1} 
and Subsection \ref{sec:fluidthms} for more alternative statements.
The proofs are given at the end of Subsection \ref{sec:fluidthms}.
The pressure terms are constructed by precomposing the distributions defined by
the fluid velocity with the Bogovskij type inverse of the divergence.
Most estimates can be written either in a sharp form using weighted Sobolev norms or in 
a more transparent form with different integrability indices but the weighted norm only on one side.
Hence there is a variety of sets of pressure estimates that can be derived from Theorem \ref{thm:fluid1}.
The ones above are an attempt to single out two of the simplest non-trivial cases,
but 
the exact information about the very weak solution at hand ultimately decides what is the estimate best possible.
In particular, 
our estimates on the time derivative part can be improved considerably 
if we assume either the domain to be more regular in time or the velocity to possess higher integrability
near the boundary.

% Both results above can also be viewed as manifestations of 
% a time dependent version of the Lions--Ne\v{c}as negative norm theorem,
% which we discuss in Section \ref{sec:pressureduality}. 

Finally, we conclude the introductory discussion by listing 
a few additional instances,
where it can be expected that the operator $B$ find further applications.

\subsection{Inhomogeneous boundary data for incompressible fluids}
The Navier--Stokes equation has been studied extensively in domains varying in time,
and we point out~\cite{fujita1970existence,bock1977navier,taylor2000incompressible} and the references therein for more background. 
In these references, 
either strong solutions are considered or 
the pressure is not constructed explicitly. 
Generally, the domains considered there are at least Lipschitz in space. 
Even in cylindrical domains and for the Stokes operator, 
pressure estimates are known to be a delicate issue 
if no strong solution is expected to exist, 
which might be the case for example due to a rough source term on the right hand side. 
See for instance the seminal paper~\cite{KocSol02}.

The results in the present paper give a construction of the pressure 
for weak Leray solutions or even very weak solutions.
 %(as is performed in Theorem~\ref{thm:fluid1} below). 
A related problem is the case of inhomogeneous boundary conditions on non-cylindrical domains.
In order to obtain a priori estimates there,
a solenoidal extension operator is needed.
We consider the Bogovskij type operator as an intermediate step towards the extension operator,
and we plan to return to this question in a later work.

\subsection{Fluid-structure interactions}
Another set of applications arises in the setting of fluid-structure interactions,
where a fluid interacting with a possibly deforming or moving solid is modeled.
The fluid velocity and pressure are then defined in a variable-in-time fluid domain. 
It is noteworthy that 
even the existence of a distributional pressure is left undiscussed
in many works on weak solutions for fluid-structure interactions. 
This includes rigid body motions~\cite{galdi2006existence,GalSil09,Gal13,DEGLT} 
as well as elastic shells interacting with fluids~\cite{LenRuz14,canic2020moving,MuhSch20}. Similarly, estimates on the distributional time derivative 
in duals of solenoidal subspaces of Sobolev spaces are generally missing in the non-cylindrical set-up.
In cylindrical domains, they are obtained directly by the concept of weak solutions.
See~\cite{SchSro20} where an estimate for the distributional time derivative was shown for weak solutions of a fluid-structure interaction problem.
 
Not only does the Bogovskij operator here allow for constructing a distributional time derivative and a pressure for these references,
it also provides a non-trivial regularity gain 
relative to the regularity of the velocity and to the fluid domain,
which we consider as input data in the present paper. 
Moreover, 
it can be used for constructing approximating sequences of test-functions, 
which is usually a very delicate technical issue 
since the space of test functions depends on the motion of the solid in that setting. 
Up to the authors' knowledge, 
so far pressure reconstruction as well as the respective solenoidal approximation results 
are only available in smooth space-time domains~\cite{BenKamSch20}. 
 
\subsection{Density estimates for compressible fluids}
Higher integrability of the density, and consequently that of the pressure,
is nowadays commonly obtained via Bogovskij operators~\cite{FeiNovPel01,BreSch18}. 
The operator introduced here allows for a significant relaxation of the assumptions on the pressure
in \cite{BreSch18,BreSch20}, where interaction of elastic shells and compressible fluids is studied.
%More precisely, the lower bound on the adiabatic exponent there can now be reduced from $12/7$ to $3/2$.

\subsection{Homogenization}
The prototype set-up for homogenization is a domain with an increasing number of holes 
with decreasing size
to represent particles that are breaking the flow. 
Ever since the seminal work of Tartar~\cite{Tat80} (appendix in the book) and Allaire~\cite{All90I,All90II}, 
so-called restriction operators that conserve solenoidality have been used in the analysis of homogenization problems. Operators of Bogovskij type may also be used in this setting~\cite{DieFeiLu17,LuSch18,HoeKowSch21}. 
Since it is reasonable that the particles move~\cite{CarrapatosoHillairet20}, 
the operator introduced here is likely to find applications for further progress in the theory of homogenization problems 
in fluids.

\subsection*{Acknowledgment} 
This work was supported by the Deutsche Forschungsgemeinschaft 
(DFG, German Research Foundation) under Germany's Excellence Strategy -- EXC-2047/1 -- 390685813 and CRC 1060. This work was also supported by the Primus research programme PRIMUS/19/SCI/01, the University Centre UNCE/SCI/023 of Charles University, the grant GJ19-11707Y of the Czech national grant agency (GA\v{C}R) as well as the ERC-CZ grant LL2105 CONTACT. S.\ S.\ wishes to thank the University of Vienna for their kind hospitality in winter 2020/21, where parts of the work were performed.

\section{Preliminaries}
\subsection{Notation}
We work in $ \R ^{1+n} $, 
the first coordinate being time and the other $ n $ being space.
Sometimes the space coordinates are split as $ x = (x',x_n) $ with $ x' \in \R^{n-1} $.
This notation is used without specific mention
and should always be understood as stated here.
The letter $ C $ is reserved for a quantity depending on data admissible for the constant 
in the statement of the proposition that is being proved.
We specify the dependency in the statements of the theorems but usually not in the proofs.
We also use the notation $a \lesssim b$ for $a \le Cb$ and $b \sim a$ if $a \lesssim b \lesssim a$.
The $ L^{p} $ norms for only space variables or both space and time variables are denoted by the same $ \no{\cdot}_p $.
We write $ \no{f(t,\cdot)}_p $ when the norm acts on space variables on a fixed slice 
and $ \no{f}_p $ when the integration is with respect to all variables.
We use the common notation $p' = p/(p-1)$ for the dual exponent of $1 < p < \infty$.
We denote the Lebesgue measures of all dimensions by $\ab{\cdot}$.
The notation $\la{f,g}$ means a duality pairing whose nature is always clear from the context,
and whenever it can be interpreted as an $L^{2}$ inner product
we will do so. Given a real number $h$, 
we define the finite difference operator $\Delta_h$  
as 
\[\Delta_h f(t,x) = f(t+h,x) - f(t,x),  \]
that is, with respect to the time variable.

\subsection{Domains}
\label{sec:domain}
The letter $ \Omega $ denotes a $ 1+n $ dimensional domain,
and $ \Omega_{t_0} $ refers to the $n$-dimensional domain 
obtained by collecting the points $x$ with $(t_0,x) \in \Omega$.
We identify the domains $ \{t\} \times S $ and $ S $ with $ S \subset \rn $
whenever this cannot cause confusion.

Most of the time,
we study domains whose boundary is locally a graph of a function.
The precise definition is as follows.

\begin{definition}[$C^{\alpha,\beta,\theta}$ domain]
  \label{def:domains}
  Let $\alpha, \beta ,\theta \in [0,1]$.
  Let $A$ be a natural number.
  A bounded, open and connected set $ \Omega \subset \R ^{1+n} $ is a
  $ C^{\alpha,\beta,\theta} $ domain with parameter $A$ if the boundary of $ \Omega $ can be covered by
  $A$ rectangles $ P_i = I_i \times R_i $ with $I_i \subset \R$ 
  such that $\partial \Omega \cap P_i $ is a graph of a function 
  that is $\alpha$-H\"older continuous in time and $ \beta $-H\"older continuous in space, 
  after a possible rotation and translation of the space coordinates. 
  In addition, 
  each time slice $\Omega_t$ is assumed to be connected.
  We use the same notation if $\alpha, \beta = 0 $
  when only uniform continuity is assumed instead of H\"older continuity as above.
  The parameter $\theta$ is as in \eqref{eq:defthetaintro}.
\end{definition}

For the readers' convenience, 
we briefly discuss the relation of $C^{\alpha,\beta,\theta}$ domains to other ubiquitous classes of domains
that we already mentioned in the introduction.
We restrict this comment to the time independent setting.
A bounded domain $ \Omega $ is $ s $-John with $ s \ge 1 $ 
if there is a center point $ z $ such that each $ x \in \Omega $
can be connected to $ z $ with a rectifiable curve $ \gamma $ 
parametrized by its arc length such that $ \gamma(0) = x $, $ \gamma(\ell(\gamma)) = z $
and 
\[
t^{s} \le C \dist(\gamma(t), \Omega^{c})
\]
for a constant $ C $ only depending on the domain and $z$.
The $1$-John domains are usually called John domains for brevity.

The $ s $-John domains and $ C^{1/s} $ domains allow the spiky exterior cusps of the same order.
Hence $ C^{\beta} $ domain need not be $ 1/(\beta + \epsilon) $ John.
On the other hand, a John domain can have a fractal boundary such as the von Koch snowflake, 
and hence even $ 1 $-John domain need not be a $ C^{0} $ domain.

For each Sobolev estimate we prove, 
there is an underpinning 
$ L^{p} $ estimate (see \eqref{eqthm:lpweight}) 
which encodes the geometry of the domain.
These $ L^{p} $ estimates are equivalent to certain weighted Poincar\'e inequalities,
compare to the argument leading to Lemma 4.1 in \cite{MR3558524}.
For those, sharpness has been studied extensively.
Our weighted $ L^{p} $ estimates \eqref{eqthm:lpweight}
specialized to the time independent case are slightly worse 
than what can be proved for domains with a single cusp,
such as 
\[
\{ (x,y) \in (-1,1): 0 < y < (1-x)^{1/\beta} \} .
\]
Our first order estimates with unweighted 
right hand side are as good as the ones proved in \cite{MR2606245},
but the case with unweighted left hand side does not allow for as good an estimate 
as what \cite{MR2606245} provides in the special case of a domain with a single cusp.
On the other hand, 
the graph type domain that fits in our framework
allows for estimates better than
what is know to be sharp for general $s$-John domains.
See \cite{MR1668136}, \cite{MR1768998} and also the book \cite{MR2777530}.
In particular, 
the sharp Poincar\'e inequality with $s = 1/\beta$ and $p=q$ in \cite{MR1768998}
is dual to sharp weighted $L^{p'}$ solvability of the divergence.
We also refer to \cite{MR2988724} for further discussion on sharpness of related results. 
The choice to treat $C^{\alpha,\beta,\theta}$ domains here 
was done in order to provide an operator for non-cylindrical space-time domains. 
This can be read from our construction 
where it seems necessary to have fixed local-in-time coordinates for the spatial domains.

\subsection{Function spaces}
\label{sec:function}
We define all the function spaces appearing in this paper 
as closures of test functions with respect to various norms.
The standard test function space is that of smooth and bounded functions,
but we also use spaces defined as closures of smooth functions with some additional properties,
such as compact support, mean zero, divergence freeness and so on.
If we impose an additional assumption on the space of test functions,
it will be denoted as a subscript such as
\begin{itemize}
  \item $0$ for $\supp \varphi$ compact,
  \item $\pi^{*}$ for $\la{ \dive \varphi(t,\cdot) ,1} = 0$ for all $t$,
  \item $smz$ for $\la{\varphi(t,\cdot),1} = 0$ for all $t$,
  \item $sol$ for $\dive \varphi(t,x)=0$ for all $(t,x)$.
\end{itemize}
This convention only applies to Sobolev spaces of non-negative order,
which we consider next.
The negative order spaces will be discussed afterwards.

Let $ k \ge 0$ be an integer and consider multi-indices $ \gamma \in \{0\} \times \N^{n} $.
Given a test function $ f $,
we define the homogeneous Sobolev norm as 
\begin{align*}
\no{f(t,\cdot)}_{\dot{W}^{k,p}} &=  \left(   \sum_{|\gamma| = k} \int |\partial^{\gamma} f(t,x)|^{p} \, dx \right)^{1/p},  
\end{align*}
We will also need a family of weighted norms,
which we define as follows. 
Consider the vectors 
\[
\aleph = (k,p,\beta,\kappa,\alpha,b), \quad \aleph' = (-k,p',-\beta,-\kappa,-\alpha,-b)
\] 
where $k$ is a non-negative integer telling the number of space derivatives,
$p \in (1,\infty)$ is the exponent of integrability,
$\beta \ge 0$ is a distortion parameter for space regularity,
$\kappa$ is the number of time derivatives,
$\alpha \ge 0$ is the distortion parameter for time regularity and 
$b \in \R $ is an additional weight parameter.
We define 
\begin{multline*}
%\label{alephnorm}
  \no{f}_{W^{\aleph}(t,\Omega)} \\=  
 \left( \sum_{\lambda = 0}^{\kappa} \sum_{l=0}^{k} \sum_{|\gamma| = l} \int |\partial^{\gamma}\partial_t^{\lambda} f(t,x)|^{p}\dist(x,\partial \Omega_t)^{ p \left(\frac{l-k}{\beta} + \frac{\lambda- \kappa}{\alpha} - b \right)} \, dx \right)^{1/p}, 
\end{multline*}
and
\[
\no{f}_{L^{q}W^{\aleph}(\Omega)} = \left(\int \no{f}_{W^{\aleph}(t,\Omega)}^{q} \, dt \right)^{1/q}.
\]
The standard Sobolev spaces fall under the scale $L^{a}W^{\aleph}$,
and the basic examples are 
\begin{align*}
  \aleph = (k,p,\infty,0,\infty,0) & \Leftrightarrow W^{\aleph}(t,\Omega) = W^{k,p}(\Omega_t), \\
  \aleph = (0,p,\infty,0,\infty,0) & \Leftrightarrow L^{q}W^{\aleph}(\Omega) = L^{q}L^{p}(\Omega), \\
  \aleph = (0,p,\infty,0,\infty,\beta) & \Leftrightarrow W^{\aleph}(t,\Omega) =  d^{p\beta} L^{p}(\Omega_t), 
\end{align*}
and the duality formula 
\[
(d^{p\beta} L^{p}(\Omega_t))^{*} = d^{-p'\beta} L^{p'}(\Omega_t)
\]
is consistent with the notation or $\aleph'$.
% However,
% the usual notation with iterated spaces 
% quickly becomes too cumbersome when considering more complicated expressions of mixed smoothness and integrability in variable domains,

For space variables only, 
we define the norms with negative smoothness directly as dual norms.
We set 
\begin{align*}
  \no{f(t,\cdot)}_{\dot{W}^{-k,p'}(\Omega_t)} &= \sup_{g \in C_0^{\infty}(\Omega_t), \ \|g\|_{\dot{W}^{k,p}} \le 1} \abs{ \int g(t,x)f(t,x) \, dx } ,\\
\no{f}_{W^{\aleph'}(t,\Omega)} &= \sup_{g \in C_0^{\infty}(\Omega_t), \ \|g\|_{W^{\aleph}(t,\Omega)} \le 1} \abs{ \int g(t,x)f(t,x) \, dx },
\end{align*}
when $\kappa = 0$.
When $\kappa  \ge 0$ and $k < 0$,
we set 
 \[
 \no{f}_{W^{k,p,\beta,\kappa,\alpha,b}(t,\Omega)}  = \left( \sum_{\lambda= 0}^{\kappa} \sup_{g } \abs{ \int g(t,x) \partial_{t}^{\lambda} f(t,x) \, dx }^{p} \right)^{1/p}
 \]
where the supremum is over all 
\[
g \in  C_0^{\infty}(\Omega_t), \quad  \|g\|_{W^{-k,p',-\beta,0,\infty,-b - (\kappa - \lambda)/\alpha }(t,\Omega)} \le 1
.\]
More generally, we define the dual norms of the remaining mixed spaces along 
\[
 \no{f}_{L^{q'}W^{\aleph'}(\Omega)}  = \sup_{g \in C_0^{\infty}(\Omega_t), \ \|g\|_{L^{q}W^{\aleph}(\Omega)} \le 1} \abs{ \int g(t,x)f(t,x) \, dx }
\]
when $1 < q < \infty$. 
Note that the values $\pm \infty$ appearing as third or fifth index give rise to the same norm.
This is consistent,
as we prefer to regard the dual space of an unweighted space as another unweighted space.
We will later abuse the notation and simply identify the values $\pm \infty$.
Note further that if a smoothness parameter, first index for spatial smoothness or fourth index for temporal smoothness, is zero, 
then the corresponding distortion parameter, 
third for space and fifth for time, becomes meaningless and can be set to $\infty$.
These notations are used in particular Theorems \ref{thmint:fluid2} and \ref{thmint:fluid3}, which
are derived very carefully from Lemma \ref{lem:distribution} at the end of the paper.
We also note that the spaces $L^{q}W^{\aleph}(\Omega)$ are not genuinely new but merely a notation 
to keep track of different weight parameters.
One could be very precise and define a Sobolev space that treats each 
mixed derivative separately with a particular weight,
but we restrict the generality to the setting as described above.

The norms of $\dot{W}_0^{-k,p'}(\Omega_t)$ and $W_0^{\aleph'}(t,\Omega)$
are defined by taking the supremum over the larger class of functions $g \in C^{\infty}(\Omega_t)$,
not necessarily compactly supported.
This diverges slightly from the notational convention for subscript indices of positive order spaces
but it attempts to capture the idea that a functional having zero boundary values 
acts on test functions without and vice versa.
Moreover, both topologies with negative Sobolev norms give the same closures 
for $C^{\infty}$ and $C^{\infty}_0$.
Hence there is no need to reserve the subscript zero to indicate 
the boundary values of the dense subspace.

Finally, 
we also use shortened notations such as 
\[
L^{q}L^{p}(\Omega), \quad L^{q}W^{k,p}(\Omega), \quad W^{\kappa,q}W^{k,p}(\Omega)
\]
and we write 
\[
\partial_t^{-1} f(t,x) = \int_{-\infty}^{t} f(s,x) \, ds 
\]
for antiderivatives in time whenever it simplifies the notation
and we can avoid using weighted $\aleph$ spaces.
The space $W^{-\kappa,q}W^{k,p}(\Omega)$ consist of functionals $f$
such that $ f \circ \partial_t^{\kappa} \in L^{q}W^{k,p}(\Omega)$.
A superscript star such as in $B^{*}$ converts an operator into its adjoint 
and a space such as in $X^{*}$ into its dual,
everything always written down a priori in terms of test functions.

\subsection{Inequalities}
For the reader's convenience,
we collect here some of the inequalities that we use repeatedly.
First, the classical Hardy's inequality
\begin{equation}
  \label{eq:hardy_classical}
  \left( \int_{0}^{\infty} \left(\frac{1}{x} \int_0^{x} |f(s)|\, ds \right)^{p} x^{b} \, dx \right)^{1/p}
  \le \frac{p}{p-b-1} \left( \int_{0}^{\infty} |f(x)|^{p} x^{b} \, dx \right)^{1/p}
\end{equation}
is
valid for all $ p > 1 $ and $ b \in (-\infty, p-1) $.
This can be understood as a weighted $ L^{p} $ bound for Hardy's operator.
The range of admissible $ b $ extends down to $ -\infty $ in contrast to the Hardy--Littlewood maximal function,
for which $ b \le -1 $ are not admissible anymore.
The difference will be important to us.
Second, we have the adjoint form of Hardy's inequality 
\begin{equation}
  \label{eq:hardy_classical2}
  \left( \int_{0}^{\infty} \left(\frac{1}{x} \int_{x}^{\infty} |f(s)|\, ds \right)^{p} x^{b} \, dx \right)^{1/p}
  \le \frac{p}{b+1} \left( \int_{0}^{\infty} |f(x)|^{p} x^{b} \, dx \right)^{1/p}
\end{equation}
which is valid for all $b \in (p-1, \infty)$.
This inequality is needed for treating the adjoint operator of the Bogovskij integral.

The next inequality is Hardy's inequality for Lipschitz domains.
Given a bounded Lipschitz domain $ O \subset \rn  $,
it holds 
\begin{equation}
  \label{eq:hardy_lipschitz}
\left( \int_O  |f(x)|^{p} \dist(x,\partial O)^{p(b-1)} \, dx  \right)^{1/p}
\le C \left( \int_O  |\nabla f(x)|^{p} \dist(x,\partial O)^{p b} \, dx  \right)^{1/p}
\end{equation}
for all test functions $ f \in C_0^{\infty}(O) $, all $ p > 1 $ and all $-\infty < b < (p-1)/p $.
The constant $ C $ only depends on the domain and the indices $ p $ and $ b $.
See for instance \cite{MR163054} or the metric space approaches in \cite{MR2506697}
and \cite{MR3168477}.
The Lipschitz condition is not necessary for Hardy's inequality to hold.
We say $O$ satisfies Hardy's inequality with negative powers if \eqref{eq:hardy_lipschitz}
holds for all $b \le 0$.
There are numerous interesting examples beyond the Lipschitz case with this property \cite{MR3168477}.
In particular, any bounded and simply connected planar domain satisfies \eqref{eq:hardy_lipschitz}
with all $b < (p-1)/p $.

We will also need Poincar\'e's inequality 
\begin{equation*}
\left( \int_Q |f(x)|^{p} \, dx  \right)^{1/p}
\le C \ell(Q)^{k} \left( \int_Q  |\nabla^{k} f(x)|^{p}  \, dx  \right)^{1/p}
\end{equation*}
which is valid for all cubes $ Q $, all $k \ge 1$
and all functions in $ C^{\infty}(Q) $
that either have compact support in $ Q $
or are $L^{2}(Q)$ orthogonal to all polynomials of degree at most $k-1$.

\section{Construction of a Bogovskij operator}
We start by defining the first reference operator.
Let $ b : \rn \to [0,1] $ be a smooth function supported in $ B(0,1) $
and having integral one.
We define the Bogovskij integral of a test function $ f $ as
\begin{equation*}
B_{B(0,1)} f (x) := \int_{\rn} f(y)(x-y) \int_{1}^{\infty} b\left( y + r(x-y)  \right) r^{n-1} \, d r d y.
\end{equation*} 
The adjoint operator of the Bogovskij integral is given through 
\begin{equation*}
B_{B(0,1)}^{*} f (x) :=  -\int_{\rn} b(y)(x-y) \int_{0}^{1} f(y+r(x-y)) \, d r d y.
\end{equation*} 

Given a generic cube $Q = Q(z,\delta)$ with center $ z $ and side length $2 \delta > 0 $,
let $ \tau (x) = (x-z)/\delta$.
We define 
\begin{equation}
  \label{eq:ref_bogo}
B_{Q}f(x) = \delta^{-1} B_{B(0,1)}(f \circ \tau)(\tau^{-1}( x )).
\end{equation}
We let $B_Q^{*}$ be the adjoint operator of $B_Q$.
It can also be written down explicitly using scaling and translation of the reference version above.
It is sometimes called the Poincar\'e integral.
We collect some fundamental properties of the Bogovskij
and Poincar\'e integrals to the following proposition.
A proof can be found for instance in \cite{MR2240056}.
More has been proved there,
but we only quote what we need.
\begin{proposition}
  \label{prop:reference}
  Let $ Q $ be a cube.
  Let $ \Omega \subset \rn $ be a bounded domain star-shaped with respect to $ Q $. 
  Then $ B_Q $ maps $ C_0^{\infty}(\Omega) $ into $C_0^{\infty}(\Omega;\rn)$ 
  and for all functions $ g \in C_0^{\infty}(\Omega) $ it holds
\[
\dive B_Q g = g - b \int g \, dx.
\]
In addition,
\[\no{B_Q}_{\dot{W}_0^{s,p}( \Omega )\to \dot{W}_0^{s+1,p}(\Omega;\rn)}
+ \no{B_Q^{*}}_{\dot{W}_0^{s,p}(\Omega;\rn)\to \dot{W}^{s+1,p}(\Omega)} < C\]
for all $ 1<p<\infty $ and $ s \ge 0$
with $C$ only depending on $s$, $p$, $n$ and $\diam (\Omega) / \diam (Q)$.
\end{proposition}

We upgrade the proposition to the case of domains consisting of finitely many cubes of equal side length. 
These domains are still time independent,
and they could be handled using any of the constructions in the literature (e.g.\ \cite{MR2240056}).
To keep the presentation self-contained,
we include the following simple ad-hoc argument based on \cite{MR2643399}.
One should notice, however,
that in many situations a better constant can be achieved
by the use of a more carefully chosen covering for the stationary construction.
 
\begin{proposition}
  \label{prop:trivial_space}
Let $ \delta > 0 $ and fix a set of points $ \{z_i\}_{i=1}^{M} \subset \delta \mathbb{Z}^{n} $ such that 
\[
S := \bigcup_{i=1}^{M} Q(z_i, \delta)
\] 
is connected. 
Denote $ S(\delta) = \{x \in \rn : \dist(x,S) < 2 \delta \sqrt{n}   \} $.

Then there is a linear operator $ B : C_0^{\infty}(S) \to C_0^{\infty}(S(\delta);\rn) $
such that for all $ f \in C_0^{\infty}(S(\delta)) $ with mean value zero
it holds
\begin{equation*}
  \dive Bf = f,
\end{equation*}
and we have the a priori estimates
\[
\|Bf\|_{\dot{W}^{s+1,p}(S(\delta);\rn)} \le C\|f\|_{\dot{W}^{s,p}(S(\delta))} 
\]
and 
\[
\|B^{*}f\|_{\dot{W}^{s+1,p}(S(\delta))} \le C\|f\|_{\dot{W}^{s,p}(S(\delta);\rn)}
\]
for all $1 <p < \infty $ and $ s \ge 0 $.
Here $ C $ only depends on $s$, $ p $, $ n $, $ \delta $ and $M$.
\end{proposition}

\begin{proof}
We decompose $ f $ into a sum of functions with mean value zero and support in $ Q(z_i,2\delta) $.
Let $ \eta_i \ge 0$ be a smooth function that is supported in $  Q(z_i,2 \delta)  $, is bounded by one
and satisfies
\[
1_{S}(x) \sum_{i=1}^{M} \eta_i(x) = 1_S(x).
\]
We define $\tilde{\eta}_i$ through 
\[
\tilde{\eta}_{i}(x) = \left( \int \eta_i(y) \right)^{-1} \eta_i(x).
\]
We call $ i $ and $ j $ neighbors if $ |z_i - z_j| \le \sqrt{n} \delta $.
Given $ z_{j} $, let $\{c_j(k)\}_{k=1}^{L(j)} \subset \{1, \ldots, M\} $ be a tuple such that $ c_j(k) $ and $ c_j(k-1) $
are neighbors for all $ k > 0 $, $ c_j(1) = j $ and $ c_j(L(j)) = 1 $.
In addition, we choose the tuple so that $ L(j) $ is minimal.
Define 
\[
T_i f = f \eta_i - \la{f, \eta_i} \tilde{\eta}  + \sum_{k=1}^{\infty} \sum_{j=1}^{M} 1_{\{j:\ i = c_j(k),\ j\ne i \} } \la{f,\eta_j} (\tilde{\eta}_{c_j(k-1)} - \tilde{\eta}_{c_j(k)}).
\]
Because the functions $ \tilde{\eta}_j $ have mean value one,
it readily follows that $ T_i f(x) $ has mean value zero.

By minimality of $ L(j) $, each $ i $ satisfies $ c_j(k) = i  $ for at most one $ k $.
Summing over $ i $ and telescoping, 
the sum above becomes
\[
  \sum_{j=1}^{M} \sum_{k=2}^{L(j)} \la{f,\eta_j} (\tilde{\eta}_{c_j(k-1)} - \tilde{\eta}_{c_j(k)}) = 
  \sum_{j=1}^{M} \la{f,\eta_j} (\tilde{\eta}_{c_j(1)} - \tilde{\eta}_{c_j(L(j))}) .
\]
Here $ c_j(1) = j $ and $ c_j(L(j)) = 1 $, 
and consequently
\[
  \sum_{i=1}^{M} T_if = f  - \eta_1 \sum_{j=1}^{M}  \la{f,\eta_j} = f.
\]

Now we have formed the decomposition.
Denote $ Q_i = Q(z_i,\delta) $
and define
\[
B f = \sum_{i=1}^{M} B_{Q_i} T_i f.
\]
Clearly $ \dive Bf = f $. 
As $ B_{Q_i}T_i f $ is supported in $ Q(z_i,2 \delta) $,
we see that $ Bf $ is supported in $ S(\delta) $.
Finally, as the overlap of $ Q(z_i,2\delta) $ is bounded by a dimensional constant,
we deduce by means of Proposition \ref{prop:reference} for $ \gamma \in \N^{n} $
\[
\|\partial^{\gamma} B f\|_p^{p}
 \le C \sum_{i=1}^{M}\|\partial^{\gamma} B_{Q_i} T_i f\|_{p}^{p}
 \le C \sum_{i=1}^{M} \sum_{|\gamma'| = |\gamma| - 1}\|\partial^{\gamma'} T_i f\|_{p}^{p} .
\]
Expanding the definition of $ T_i $,
applying Leibniz rule,
using that $ \eta_i $ have bounded derivatives 
and finally applying Poincar\'e's inequality for trace zero functions in $S$,
we conclude the claimed bounds for Sobolev spaces of integer order.
The bounds on fractional Sobolev spaces follow by interpolation
although we do not need them here.
This concludes the proof for $B$.

The adjoint operator is defined through 
\begin{equation*}
T_i^{*}g = g \eta_i - \la{g, \tilde{\eta}_i} \eta_i  + \sum_{k=1}^{\infty} \sum_{j=1}^{M} 1_{\{j:\ i = c_j(k),\ j\ne i \} }    \la{g,\tilde{\eta}_{c_j(k-1)}-\tilde{\eta}_{c_j(k)}}\eta_j
\end{equation*}
and 
\begin{equation*}
B^{*} g = \sum_{i=1}^{M} T_i^{*} B_{Q_i}^{*} g.
\end{equation*}
We may notice \seb(by the definition of $T^*$) that $B^{*}g$ is supported in $S(\delta)$.
Using the Leibniz rule, Poincar\'e's inequality and Proposition \ref{prop:reference},
we conclude the desired bounds for $B^{*}$.
\end{proof}

The simple operators above can be extended to operators acting on space-time functions by 
a trivial extension,
that is,
as operators that ignore the time variable.
Such an operator is used to handle space-time cylinders.
To deal with shapes more general than constant in time cylinders,
we weld together thin cylindrical domains.
The operator living in the resulting domain can be regarded as an error term
in the main construction later.
The relevant estimates are easy to prove, 
but we still ignore the fine boundary behavior of the space-time domain,
which will later be the main problem.

\begin{proposition}
  \label{prop:interior_bogovski}
Let $ \delta > 0 $ and let $ T > 0 $ be an integer.  
For each $ j \in \{1,\ldots,T\} $, 
let $ S_j $ be a finite union of $ \delta $-cubes as in Proposition \ref{prop:trivial_space}.
Denote $ I_j = [\delta(j-1),j\delta)$, 
$ S_{I_j} = I_j \times S_j $, and 
\[
\Xi = \bigcup_{j=1}^{T} S_{I_j} .
\]
Set $ \Xi(\delta) = \{(t,x) \in \R ^{1+n}: \dist(x,S_j) < 2 \delta \sqrt{n} ,\ t \in \bigcup_{k=j-1}^{j+1} I_k \}  $.

Then there exists a linear operator $ B : C_0^{\infty}(\Xi) \to C_0^{\infty}(\Xi(\delta);\rn) $
such that for all $ f \in C_{0,smz}^{\infty}(\Xi) $ 
\[
  \dive Bf = f
\]
and the following a priori estimates hold.
For all $1 <p < \infty $, $ \lambda \ge 0 $ an integer, $ s \ge 0$, $ h \to 0 $ and all $ t $
\begin{align*}
\no{\partial_t^{\lambda} Bf(t,\cdot)}_{\dot{W}^{s+1,p}(\Xi(\delta);\rn)} & \le C \sum_{\lambda'=0}^{\lambda} \no{\partial_t^{\lambda'} f(t,\cdot)}_{\dot{W}^{s,p}(\Xi(\delta))} \\
\no{ \Delta_{h} Bf(t,\cdot) }_{\dot{W}^{s+1,p}(\Xi(\delta);\rn)} & \le C \sum_{\lambda'=0}^{\lambda} \no{ \Delta_{h}f(t,\cdot) }_{\dot{W}^{s,p}(\Xi(\delta))} + O(h).
\end{align*}
The adjoint operator $B^{*}$ satisfies 
\begin{align*}
\no{\partial_t^{\lambda} B^{*}f(t,\cdot)}_{\dot{W}^{s+1,p}((\Xi(\delta)))} & \le C \sum_{\lambda'=0}^{\lambda} \no{\partial_t^{\lambda'} f(t,\cdot)}_{\dot{W}^{s,p}(\Xi(\delta);\rn)} \\
\no{ \Delta_{h} B^{*}f(t,\cdot) }_{\dot{W}^{s+1,p}(\Xi(\delta))} & \le C \sum_{\lambda'=0}^{\lambda} \no{ \Delta_{h}f(t,\cdot) }_{\dot{W}^{s,p}(\Xi(\delta);\rn)} + O(h).
\end{align*}
Here $ C $ only depends on $ p $, $ s $, $ l $, $ n $ and the data of $\Xi$ 
(referring to all constants from Proposition \ref{prop:trivial_space}) but not on $ t $, $ f $ or $ h $.
\end{proposition}

\begin{proof}
Take smooth functions $ \eta_j : [0,\delta T] \to [0,1] $ 
such that $\supp \eta_j \subset [\delta(j-3/2),\delta(j+1/2)] \cap [0,\delta T] $ and
\[
 \sum_{j=1}^{T} \eta_j(t) = 1_{[0,\delta T]}.
\]
Let $ B_j $ be the operator from Proposition \ref{prop:trivial_space} relative to $ S_j $.
It is then easy to see that 
\[
B f = \sum_{j=1}^{T} \eta_j B_j f
\]
and its adjoint
satisfy the assertions of the proposition.
\end{proof}

Now we are ready with the interior piece.
We can turn the attention to the boundary.
The idea is to use the graph structure in order to exploit line integrals as opposed to curves that appear in many time independent constructions.
The explicit description of the curve family of straight lines
is easy to handle when the domain changes in time.
Less explicit constructions would quickly become intractable.

\begin{proposition}
  \label{prop:bdry_bogovski}
  Let $ R $ be an $ n - 1 $-dimensional rectangle, 
$ I $ an interval, $ \beta \in [0,1] $ 
and $ \psi : I \times R \to (0,\infty) $ a function which is $ \beta $-H\"older continuous (uniformly continuous if $ \beta = 0 $) with respect to space
  and measurable with respect to time.
  Let 
  \begin{align*}
  P &= \{ (t,x',x_n) \in I \times R : x_n < \psi(t,x') \},  \\
  d(t,x) &= \dist(x,\{(\xi',\psi(t,\xi')) : \xi' \in R \} ).
  \end{align*}
  Then there exists a linear operator $ B $ on $ C^{\infty} (P) $
  such that for all test functions $ f $
  \[
  \dive Bf = f, \quad Bf \in C^{\infty}(P;\rn)
  \]
  and the a priori estimates 
  \begin{align*}
    \int \left( \frac{|Bf(t,x)|}{d(t,x)^{\beta}} \right)^{p} d(t,x)^{\beta b}  \, dx
    & \le C \int |f(t,x)|^{p} d(t,x)^{\Theta b} \, dx \\
     \sum_{\substack{\gamma \in \{0\} \times \N^{n} \\ |\gamma| = k }} |\partial^{\gamma} Bf(t,x)| &\le 
     \sum_{\substack{\gamma \in \{0\} \times \N^{n} \\ |\gamma| = k }} |B\partial^{\gamma} f(t,x)| 
    +  \sum_{\substack{\gamma \in \{0\} \times \N^{n} \\ |\gamma| = k-1 }} |\partial^{\gamma} f(t,x)| 
  \end{align*}
hold with 
\[
\Theta = \begin{cases}
  1, & \quad  \trm{if $ b \le 0$}\\
  \beta, & \quad \trm{if $0 < b  \le p$ and $ p < (p-1) /\beta$}\\
  \beta^{2}, & \quad  \trm{if $p <  b < (p-1)/\beta$ or $0 < b < (p-1)/\beta \le p$}
\end{cases}.
\]
uniformly in $ t $ and for all integers $k \ge 1$.
The constant satisfies 
\[
C = C(n,p,b) \|\psi \|_{C^{\beta}}
\]
and $ B $ commutes with $\partial_t $ and finite differences $ \Delta_h $ in the time variable.

If it holds in addition that $ f \in C_0^{\infty}(P) $,
then there also exists an $ \epsilon(f) > 0 $ 
such that $ Bf = 0 $ whenever $\psi(t,x') -  x_n < \epsilon$.
\end{proposition}
 
\begin{proof}
By rotational and translational invariance of the operators involved,
we may assume that $R = \prod_{i=1}^{n} [0,r_{i}]$.
Using these coordinates,
we set
\[
Bf(t,x',x_n) = \left(0, \ldots, 0, - \int_{x_n}^{\infty} f(t,x',s) \, ds \right).
\]  
Then it is obvious that $ \dive Bf = f $.
The existence of $ \epsilon(f) $ as claimed is also clear from the formula.
Commutation with finite differences and derivatives in time variable is also immediate
as is the pointwise bound for high order derivatives.
It remains to prove the $L^{p}$ bound.

Because $ \psi $ is $\beta$-H\"older continuous,
we see that 
\[
d(t,x) \le |x_n - \psi(t,x')| \le  C d(t,x)^{\beta}
\] 
with a constant only depending on the dimension and the H\"older norm.
Then 
\[
|B(t,x)| \le C d(t,x)^\beta \abs{ \frac{1}{\psi(t,x')-x_n} \int_{0}^{\psi(t,x')-x_n} f(t,x', \psi(t,x')-s ) \, ds } .  
\]
After a change of variable in the $ x_n $ coordinate,
the expression inside the absolute value becomes the classical Hardy operator 
\[
Ag (s) =  \frac{1}{s} \int_{0}^{s} |g(s)| \, ds
\]
that satisfies the bound \eqref{eq:hardy_classical}
for all $ b < p-1 $.
Applying this in $ x_n $ variable,
we conclude by Fubini's theorem
\begin{multline*}
\int  \left(\frac{|Bf(t,x)|}{d(t,x)^{\beta}} \right)^{p}d(t,x)^{\beta b} \, dx
\\
\leq C\iint  \left(\frac{|Bf(t,x',x_n)|}{|x_n - \psi(t,x')|}\right)^{p}|x_n - \psi(t,x')|^{b} \, dx
 \\
\le C \iint  |f(t,x',x_n)|^{p} |x_n - \psi(t,x')|^{ b} \, dx_n dx 
\end{multline*}
whenever $b \le p$.
If $b \le 0$,
we conclude the case $\Theta = 1$.
If $0< b \le p $ with $b < (p-1)/\beta$, we conclude the case $\Theta = \beta$.
If $p < b $, 
we can directly estimate 
\[
d(t,x)^{\beta(b-p)} \le |x_n - \psi(t,x')|^{\beta(b-p)}, \quad |x_n - \psi(t,x')|^{\beta b} \le d(t,x)^{b\beta^{2}}
\]
instead 
and apply Hardy's inequality to conclude the bound 
\[
\int \left( \frac{|Bf(t,x)|}{d(t,x)^{\beta}} \right)^{p} d(t,x)^{b\beta}  \, dx
    \le C \int |f(t,x)|^{p} d(t,x)^{b\beta^{2}} \, dx
\]
whenever $b < (p-1)/\beta$
to get the remaining case $\Theta = \beta^{2}$.
\end{proof}

\begin{proposition}
\label{prop:adjointtrivial}
Consider the adjoint $B^{*}$ of the operator from Proposition \ref{prop:bdry_bogovski}. 
Then $B^{*}:C^\infty(\Omega,\rn)\to C^\infty(\Omega)$ commutes with differentiation and finite differences in time 
and the a priori estimates 
\begin{align*}
    \int |B^{*}f(t,x)|^{p} d(t,x)^{b\beta}   \, dx
    &\le C \int |f(t,x)|^{p} d(t,x)^{\Theta b+ \beta p} \, dx \\
    \sum_{\substack{\gamma \in \{0\} \times \N^{n} \\ |\gamma| = k }} |\partial^{\gamma} B^{*}f(t,x)| &\le 
    \sum_{\substack{\gamma \in \{0\} \times \N^{n} \\ |\gamma| = k }} |B^*\partial^{\gamma} f(t,x)| 
    + \sum_{\substack{\gamma \in \{0\} \times \N^{n} \\ |\gamma| = k-1 }} |\partial^{\gamma} f(t,x)| 
  \end{align*}
hold with 
\[
\Theta = \begin{cases}
  \beta, & \quad \trm{if $-1 < b < 0 $}\\
  \beta^{2} , & \quad  \trm{if $0 \le  b$}
\end{cases}.
\]
If $\beta = 0$,
the bounds for $b \ge 0$ still holds.
\end{proposition}
\begin{proof}
The proof follows from the same computation as above.
However, instead of \eqref{eq:hardy_classical} we use the adjoint form \eqref{eq:hardy_classical2}.
Note that the case $\beta = 0$ only uses the bound 
\[
d(t,x) \le |x_n - \psi(t,x')| \le C
\]
which is true without assumptions on continuity.
\end{proof}

Now we are in a position to put together the pieces and define a primitive Bogovskij operator in space-time domains with
a graph boundary.
We can use the primitive Bogovskij operator to decompose generic input data $ f $ into localized pieces $ f_j $,
to which we can apply the theory for circular cylinders, Proposition \ref{prop:reference}.

It is only the latter step that will make the operator map $ L^{p} $ to a first order Sobolev space.
Indeed, the primitive operator will not be well-behaved in any space involving Sobolev regularity.
As such, it only satisfies $ L^{p} $ bounds with a weight depending on the spatial H\"older regularity of the boundary.
While the estimates for the primitive operator hold in very general domains,
the more and more refined estimates require more and more boundary regularity from the domain.

The previously mentioned primitive operator 
corresponds to $ C^{0,0,\theta} $ domains.
The domains $ C^{0,\beta} $ host operators that are Sobolev regular in space.
In order to prove meaningful estimates with time derivatives,
we have to assume $C^{\alpha,\beta,\theta}$ regularity.
Lipschitz domains admit particularly simple estimates,
and there the construction reproduces the known stationary results,
if trivial time dependency is assumed.

In what follows we constantly assume $\Omega$ be a $C^{0,0,\theta}$ domain.
It comes with a distance map $d(t,x) = \dist(x, \partial \Omega_t)$,
whose all powers we extend by zero as $1_{\Omega}d(t,x)^{b}$ without further notice.
We also note that $d$ is uniformly bounded from above as our domain is,
and its fibers $t \mapsto d(x,t)$ 
are uniformly continuous as can be seen from the regularity assumption on the boundary of the domain.
\begin{theorem}
  \label{thm:bogovski_on_w1p}
  Let $ \alpha, \beta \in [0,1] $. 
  Assume that $ \Omega \subset \R ^{1+n} $ is a $C^{\alpha,\beta,\theta}$ H\"older domain.
  Then there exists a linear operator $ B $ acting on $C_{smz}^{\infty}(\Omega)$
 such that for all $f \in C_{smz}^{\infty}(\Omega)$ 
  \begin{itemize}
    \item $Bf$ is $C^{\infty}(\Omega)$ smooth;
    \item it holds $\dive Bf = f$;
    \item if $f$ is compactly supported in $\Omega$ so is $Bf$.
  \end{itemize}

  Let $ 1<p< \infty $, let $b\in \R$ and set 
\[
\Theta = \begin{cases}
  1, & \quad  \trm{if $ b \le 0$}\\
  \beta, & \quad \trm{if $0 < b  \le 1$ and $ p < (p-1) /\beta$}\\
  \beta^{2}, & \quad  \trm{if $1 <  b < (p-1)/\beta$ or $0<b <(p-1)/\beta \le p$}
\end{cases}.
\]
Then there exists a finite constant $ C $ only depending on $ \Omega $, $\delta$, $b$, $ p $, $\lambda$ and $ n $
such that for all times $ t $ and all compactly supported test functions $f$ 
\begin{align}
  \label{eqthm:lpweight}
  \|d(t,\cdot)^{\beta(b-1)} Bf(t,\cdot)\|_p  &\le C  \|d(t,\cdot)^{\Theta b} f(t,\cdot) \|_p ,\\
  \label{eqthm:lpcont}
  \|d(t,\cdot)^{\beta(b-1)} \Delta_h Bf(t,\cdot)\|_p  &\le C \|d(t,\cdot)^{\Theta b} \Delta_h f(t,\cdot) \|_p + o(1)
\end{align}
where the asymptotic estimate holds when $ h \to 0 $.
The estimate \eqref{eqthm:lpweight} holds even without the assumption on compact support.

If $ \beta > 0$, then for all compactly supported test functions
\begin{align}
  \label{eqthm:sobolev}
  \|d(t,\cdot)^{\beta(b-1) + 1} \nabla  Bf(t,\cdot)\|_p &\le C \|d(t,\cdot)^{\Theta b}f(t,\cdot)\|_p ,\\
  \label{eqthm:sobcont}
  \|d(t,\cdot)^{\beta(b-1) + 1} \nabla  \Delta_h Bf(t,\cdot)\|_p &\le C \|d(t,\cdot)^{\Theta b}\Delta_h f(t,\cdot)\|_p + o(1)
\end{align}
where the asymptotic estimate holds when $ h \to 0 $.
The estimate \eqref{eqthm:sobolev} holds even without the assumption on compact support.

If $\alpha > 0$, then
\begin{align}
  \label{eqthm:time}
  \|d(t,\cdot)^{\beta(b-1)}\partial_t^{\lambda} Bf(t,\cdot)\|_p &\le C \sum_{\lambda'=0}^{\lambda} \|d(t,\cdot)^{\Theta b - (\lambda- \lambda')/(\alpha \beta)}\partial_t^{\lambda'} f(t,\cdot)   \|_p 
\end{align}
holds for all test functions, not necessarily compactly supported.

Finally, if any of the right hand sides above is finite,
then $Bf$ is in the closure of $C_0^{\infty}(\Omega;\rn)$ with respect to the norm appearing as the relevant left hand side.

\end{theorem}
\begin{proof}
  We divide the proof into five subsections.
  In Subsection~\ref{sec:auxop}, we construct the first auxiliary operator,
  which satisfies weighted $ L^{p} $ bounds
  but not much more.
  This follows by just gluing together the objects constructed so far.
  Then in Subsection~\ref{sec:whitney} we form an adapted space-time decomposition of the domain,
  which together with the auxiliary operator is used to decompose the input data into localized pieces.
  This is used in Subsection~\ref{sec:construction} to define the operator from Theorem~\ref{thm:bogovski_on_w1p} locally. This is followed by verifying the claimed bounds.

\subsection{The first auxiliary operator}
\label{sec:auxop}
By assumption on the domain, we can cover the boundary by finitely many rectangles $ R_i = I_i \times Q_i $.
Each rectangle comes with a normal vector $ \vec{n}_i $, which gives the direction of the graph coordinate, 
and this vector is normal to the hyperplane containing  $ Q_i $.
Without loss of generality, we may assume that 
the rectangles $ R_i + 10 \delta  \vec{n}_i$ would still cover the boundary for a suitably small $ \delta $.
Let 
\[
\eta_i \in C^{\infty}\left( \Omega \cap \bigcup_{j} R_j \right)
\] 
form a partition of unity subordinate to $ R_i $ in the boundary region $ \Omega \cap \bigcup_{i} R_i $,
and extend these functions by zero to functions defined in all of $\Omega$.
Let $ H_i $ be the half-space normal to $ \vec{n}_i $
with $ |\partial H_i \cap \partial R_i \cap \Omega| > 0 $
and $R_i \subset H_i$.
In other words, it contains the bottom of the rectangle $ R_i $.
Let $ \varphi_i $ be a space mollification of a piecewise affine function 
that is zero in $ \partial H_i - a \vec{n}_i $ for $a \le 0$,
one in $ H_i + 5 \delta \vec{n}_i $ 
and satisfies $ |\nabla \varphi_i| \le 1/ ( 5 \delta) $.
We also require $\partial_t \varphi_i = 0 $.
Finally, we notice that there is a domain $ \Xi $ as in Proposition \ref{prop:interior_bogovski}
such that 
\[
\Omega \setminus \bigcup_{i} R_i \subset \Xi \subset \{(t,x) \in \Omega: d(t,x) > 4 \delta \}.
\]

Let $ B_{int}$ be the operator relative to $ \Xi $ provided by Proposition \ref{prop:interior_bogovski}.
Let $ B_{ext,i}$ be the operator relative to $ R_i $ provided by Proposition \ref{prop:bdry_bogovski}.
Here we use the fact that the divergence is a translation and rotation invariant operator.
We define 
\begin{equation}
  \label{eq:bprime}
B' f = \sum_{i} \varphi_i B_{ext,i}(\eta_i f) + B_{int}\left(f- \sum_{i} \varphi_i \eta_i f - \sum_{i} \nabla \varphi_i \cdot B_{ext,i}(\eta_i f)  \right).
\end{equation}
To see that the second term is well-defined,
we have to check that the expression in the parenthesis is supported in $ \Xi $
and has mean value zero on each time slice.
The first property follows from the fact that $ \eta_i $ partition the unity in the union of $ R_i $
and that $ \varphi_i $ is identically one in $ \Xi^c \cap \supp \eta_i $.
Indeed, these facts imply 
\[
 f- \sum_{i} \varphi_i \eta_i f = 0, \quad \max_{i} | \nabla \varphi_i| = 0
\]
everywhere in $\Xi^{c} $.
To check the mean value,
we integrate by parts to see
\[
-\int \nabla \varphi_i \cdot B_{ext,i}(\eta_i f) \, dx
  =  \int \varphi_i \dive B_{ext,i}(\eta_if) \, dx
  =  \int \varphi_i \eta_i f \,dx .
\]
Hence the last two terms cancel
and we are left with the integral of $ f $. 
That one is zero by assumption.
Once we know that $ B' f$ is well-defined, 
it follows from the product rule and Propositions \ref{prop:bdry_bogovski} and \ref{prop:interior_bogovski}
that $ \dive B' f = f $.

Next we verify the relevant bounds for \eqref{eq:bprime}.
We first claim
\begin{equation}
\label{decops}
\int \left(  \frac{|B' f(t,x)|}{d(t,x)^{\beta}} \right) ^{p}d(t,x)^{\beta b} \, dx \le C \int |f(t,x)|^{p}d(t,x)^{b} \, dx
\end{equation} 
Consider first the interior term
\[
B_{int}\left(f- \sum_{i} \varphi_i \eta_i f - \sum_{i} \nabla \varphi_i \cdot B_{ext,i}(\eta_i f)  \right).
\] 
By the uniform lower bound on $d(t,x)$ 
in the support of the quantity in the display above,
we estimate the $L^{p}(d^{\beta(b-p)})$ norm  
by the unweighted $L^{p}$ norm.
Then we can apply Proposition \ref{prop:interior_bogovski}
to reduce the matter to estimating the unweighted $L^{p}$ norm of
\[
f_0 = f- \sum_{i} \varphi_i \eta_i f - \sum_{i} \nabla \varphi_i \cdot B_{ext,i}(\eta_i f) .  
\] 
The first two terms are trivially bounded by the desired quantity.
To estimate the third term, 
we notice that $d(t,x) \sim 1$ in the support of $\nabla \varphi_i$.
Hence this term reduces to the weighted bound on $B_{ext,i}(\eta_i f)$.  
To estimate the weighted $L^{p}$ norm of this,
we apply Proposition \ref{prop:interior_bogovski}. 
Hence \eqref{decops} has been proved.

The second claim is
\begin{equation}
\label{decopt}
    \int \left(  \frac{|\partial_t^{\lambda} B' f(t,x)|}{d(t,x)^{\beta}} \right) ^{p}d(t,x)^{\beta b} \, dx 
    \le C \sum_{\lambda'=0}^{\lambda}\int | \partial_t^{\lambda'} f(t,x)|^{p}d(t,x)^{b } \, dx.
\end{equation}
for $\lambda \ge 0$ an integer.
This is immediate for the $ B_{int} $ part in \eqref{eq:bprime}
by the arguments as above.
For the $ B_{ext} $ part, we note that $ B_{ext,i} $ commutes with the time derivative,
and the inequality \eqref{decopt} follows by applying Proposition \ref{prop:bdry_bogovski}
to $\partial_t^{\lambda}(\eta_i f) $.

The third bound to be verified is 
\begin{equation}
  \label{decopdif}
  \int \left(  \frac{|\Delta_h B' f(t,x) |}{d(t,x)} \right) ^{p}d(t,x)^{\beta b} \ddd   
  \le C  \int |\Delta_h f(t,x)|^{p}d(t,x)^{b } \ddd +o(1)    
\end{equation}
for $0<h< h_0(f)$ when $f$ is assumed to be compactly supported.
Because of the compact support
as well as the uniform continuity of the graphs enclosing $ \Omega $,
there is $ h_0 $ depending on the function $f$ such that if $ 0<h<h_0 $,
then $ d(t+h,x) \sim d(t,x) $ in $ \supp f$ and Proposition \ref{prop:bdry_bogovski}
can be applied to $ \Delta_h f$. 
Analogous bound for the $ B_{int} $ part in \eqref{eq:bprime} is immediate from Proposition \ref{prop:interior_bogovski}.
The distance weight is bounded away from zero by a constant in the support of that term.

\begin{remark}[Dual of $B'$]
\label{rmk:dualbprime}
For future reference,
we also note that 
we can write down formulas for operators $B_{int,\lambda'}^{*}$ with 
\[
\la{\partial_t^{\lambda} B_{int} f_0 ,  g } = \sum_{\lambda' = 0}^{\lambda} \binom{\lambda}{\lambda'}  \la{\partial_t^{\lambda - \lambda'}  f_0 , B_{int,\lambda'}^{*}g } .
\]
Let $\zeta$ be a smooth and positive bump that is compactly supported in $\Omega$
and equal to one in the support of $f_0$.
We can choose $\zeta$ such that $d(t,x) \sim 1$ in its support, 
and then
\begin{equation}
\label{decopsadjoint}
\int  |\partial^{\gamma} (\zeta B_{int,\lambda'}^{*}g)|^{p}d(t,x)^{b} \, dx \le C \sum_{|\gamma'|< |\gamma|} \int | \partial^{\gamma'} g(t,x)|^{p}  \, dx
\end{equation} 
for all $b \in \R$ by Proposition \ref{prop:interior_bogovski}.
This inequality is not needed for the current proof, 
but we will refer back to this from the proof of Theorem \ref{thm:negat}.
\end{remark}

\subsection{Decomposition of the domain}
\label{sec:whitney}
We form a decomposition of Whitney type.
By uniform continuity assumption, 
there is $ \delta: (0,1) \to (0,1)$ such that if $ |t-s| < \delta(\epsilon) $,
then $|d(t,x) - d(s,x)|< \epsilon/M $,
where $ d(t,x) $ is the distance to nearest point in $\partial \Omega $ with the same time coordinate $t$
and $M$ is a large dimensional constant that can be specified later.

Consider first all time slices $ \Omega_t $ of $ \Omega $.
For each $n$-dimensional slice domain $ \Omega_t $, 
form its Whitney decomposition by disjoint dyadic cubes $ Q^{t} $.
They are the dyadic cubes such that 
their concentric dilates $ M' Q^{t} $ are contained in $ \Omega_t $ 
but their $4M'$ dilates are not.
Here we can freely choose $1 < M' < M$.
The Whitney cubes form a partition of $\Omega_t$.
Given such a cube $ Q^{t} $, let $ r(Q^{t}) $ be its side length.
Set a height $ h(Q^{t}) = \delta( r(Q^{t})) $,
and extend $ Q^{t} $ to a cylinder 
\[
P^{t} = [t-h(Q^{t}),t+h(Q^{t})] \times Q^{t}.
\]
The family of Whitney cubes $\{Q_{j}^{t}\}$
gives rise to a family of cylinders $P_{j,t}$,
which covers the slice $ \Omega_t $,
now regarded as a subset of the space-time.

The full family $ \{P_{j,t}\}_{j \in \N,t \in \R} $ covers the whole domain $ \Omega $.
Fix a spatial side length $2^{k}$.
Fix a dyadic cube $Q \subset \rn$ with $r(Q) = 2^{k}$.
Consider the family of all $(j,t)$
such that there is an interval $I_{j,t}$ so that $I_{j,t} \times Q = P_{j,t}$ 
for a cylinder as above.
By definition, all $I_{j,t}$ are equally long.
If the interval covered by $ I_{j,t} $ is open,
we add additional intervals to contain the endpoints in their closures.
Keep a minimal subfamily of $ I_{j,t} $ 
that still covers the union of the intervals in the original family
and discard the rest.
Such a minimal subfamily has overlap bounded by two.
Performing the reduction for each dyadic cube of each side length,
we obtain a reduced family of space-time cylinders $\{ P_j \}$.

Consider further the reduced subfamily.
Let $(t,x) \in P_j$.
Let $ r_j $ be the side length of $ P_j $.
Then 
\[
  |d(t,x) - d(s,x)|  \le 2r_j / M
\]
for all $  s  $ with $ (s,x) \in P_j $ and consequently by the triangle inequality
\[
\left(M' - 1 - 2/M\right)r_j \le d(s,x) \le \left(M' + 1 + 2/M \right)r_j
\]
for all $(s,x) \in P_{j}$.
If $P_{j'}$ is another cylinder from the reduced family such that $P_j \cap P_{j'} \ne \varnothing$,
then $r_j \sim r_{j'}$.
We conclude that the family of reduced cylinders has bounded overlap,
and the same goes for the family of concentric dilates by factor less than $M'$.
We denote $ P_j^{*} = M' P_j $.
We denote the spatial bases of $ P_j^{*} $ by $ Q_j^{*} $.
The largest possible dilation that still allows for this can be changed by increasing the parameters $ M $ and $M'$ 
in the construction.

  \subsection{Construction of the operator}
  \label{sec:construction}
  Let $ \chi_j $ be a smooth partition of unity subordinate to $P_j^{*}$.
  Without loss of generality,
  we can assume the partition to satisfy
  \[
    |\nabla \chi_j(t,x)| \lesssim r_j,  
    \quad   |\partial_t \chi_j(t,x)| \lesssim \delta(r_j)
  \]
  and $ \supp \chi_j \subset P_j^{*} $.
  We let 
  \[
   f_j = \dive ( \chi_j B' f ) .
  \]
  Then
  \[
  \sum_j f_j = f, \quad \sup_{t} \left \lvert \int f_j(t,x)\, dx \right \rvert = 0,
\quad \supp f_j \subset P_j^{*}
  \]
  and we define 
  \[
  Bf = \sum_{j} B_{Q_j} f_j
  \]
  where we use the simple reference Bogovskij integral \eqref{eq:ref_bogo} acting on the space variables only.

  \subsection*{Sobolev regularity in space}
  The Sobolev bounds for \eqref{eq:ref_bogo} are known, 
  see Proposition \ref{prop:reference}.
  Fix $ t $.
  We have 
\begin{equation*}
  \int | \nabla B_{Q_j}f_j(t,x)|^{p} \, dx
  \le C \int |f_j(t,x)|^{p} \, dx 
\end{equation*}
so that by $ 1_{P_j^{*}}(t,x) d(t,x) \sim  1_{P_j^{*}}(t,x)  r_j $
\[
  \int  d(t,x)^{(1-\beta)p + b\beta} |\nabla Bf(t,x)|^{p} \, dx
  \le C \sum_{j} r_j^{(1-\beta)p + b\beta} \int |f_j(t,x)|^{p} \, dx .
\]
Here 
\[
f_j =   \nabla \chi_j \cdot B' f + \chi_j f 
\]
with $ |\nabla \chi_j| \le C r_j^{-1} $,
and the support of $ B_{Q_j} f_j $
is contained in the convex hull of the union of $ \supp f_j $ and $ P_j $.
Hence
  \begin{multline*}
    %\label{eq:sobtounw}
    \sum_{j} r_j^{(1-\beta)p + b\beta} \int |f_j(t,x)|^{p} \, dx \\
    \le C \sum_{j} r_j^{(1-\beta)p + b\beta} \left(  r_j^{-p} \int_{Q_j^*} |B' f(t,x)|^{p} \, dx  + \int \chi_j(t,x) |f(t,x)|^{p}\, dx  \right) \\
    \le C \int \left [ \left(  \frac{|B' f(t,x)|}{d(t,x)^{\beta}} \right) ^{p} +  |f(t,x)|^{p} \right]d(t,x)^{b\beta} \, dx \\
     \le \int |f(t,x)|^{p} d(t,x)^{b} \, dx
  \end{multline*}
  where we used \eqref{decops} and the fact that $ \Omega $ is bounded.
  This proves \eqref{eqthm:sobolev}.

To prove \eqref{eqthm:sobcont},
we note as before
\begin{equation*}
  \int | \nabla B_{Q_j}\Delta_h f_j(t,x)|^{p} \, dx
  \le C \int |\Delta_h f_j(t,x)|^{p} \, dx 
    = \I + \II + \III,    
\end{equation*}
where
\begin{align*}
  \I &=  C \int |\nabla \chi_j(t+h,x) -\nabla \chi_j(t,x)|^{p}|B' f(t+h,x)|^{p}\, dx \\
  \II &= C \int |\nabla \chi_j(t,x)|^{p}|B' f(t+h,x)- B' f(t,x)|^{p} \, dx \\
  \III &= C \int |\Delta_h(\chi_jf)(x,t) |^{p} \, dx.
\end{align*}
Using 
\begin{equation*}
  |\nabla \chi_j(t+h,x) -\nabla \chi_j(t,x)| \le C h \delta(r_j)^{-1}r_j^{-1}
\end{equation*}
and \eqref{decops}
for $ \I $
and \eqref{decopdif} for $ \II $,
we conclude \eqref{eqthm:sobcont}.
Note that as we prove an asymptotic estimate when $ h \to 0 $
and $ f $ is compactly supported so that $ \delta(r_j) $ has uniform lower bound,
we can assume $ h \delta(r_j)^{-1} < h^{1/2} $.

\subsection*{Bounds without space derivatives}
The remaining inequalities follow by the same reasoning.
Indeed, because $ B_{Q_j}f_j $ has compact support in $ Q_{j}^{*} $, 
by Poincar\'e's inequality for functions with zero boundary values in space variable
\[
  \int_{Q_j^{*}} |\partial_t^{\lambda} B_{Q_j}f_j(t,x)|^{p} \, dx 
  \le C  \int_{Q_j^{*}}  d(t,x)^{p} |\partial_t^{\lambda} f_j(t,x)| ^{p} \, dx 
\]
where the Whitney property of $ P_j^{*} $ gives $r_j \sim d(t,x) $
and $\lambda \ge 0$ is any integer.

We write 
\[
\partial_t^{\lambda} f_j = 
  \sum_{\lambda' = 0}^{\lambda} \binom{\lambda}{\lambda'} ( \nabla\partial_t^{\lambda'} \eta_j \cdot\partial_t^{\lambda - \lambda'} B' f +\partial_t^{\lambda'} \eta_j \cdot\partial_t^{\lambda - \lambda'} f  )
\]
to obtain the bound 
\[
|\partial_t^{\lambda} f_j| \le C \sum_{\lambda'= 0}^{\lambda} \delta(r_j)^{- \lambda'} (r_{j}^{-1}\partial_t^{\lambda - \lambda'} B' f +\partial_t^{\lambda - \lambda'}f ).
\]
Notice that if $\lambda = 0$, 
there is no factor $\delta(r_j)$
and if $\lambda > 0$,
then we are assuming $\alpha > 0$ so that $\delta(r_j) \le C r_{j}^{1/\alpha}$.
We find
\begin{align*}
&\int_{Q_j^{*}}d(t,x)^{\beta (b-1)p} |\partial_t^{\lambda} B_{Q_j}f_j(t,x)|^{p} \, dx 
\sim r_j^{\beta(b-1)p}\int_{Q_j^{*}} |\partial_t^{\lambda} B_{Q_j}f_j(t,x)|^{p} \, dx 
\\
&\quad \leq C \sum_{\lambda'= 0}^{\lambda}\int_{Q_j^{*}} d(t,x)^{\beta(b-1)p+\frac{(\lambda'-\lambda)p}{\alpha}}(|\partial_t^{\lambda'} B' f| +d(t,x)|\partial_t^{\lambda'}f|)^p\, dx.
\end{align*}
Hence
we can sum up the $Q_j^*$ use there finite overlap and use \eqref{decopt}
to conclude \eqref{eqthm:lpweight} and \eqref{eqthm:time}.
Similarly, the asymptotic estimate \eqref{eqthm:lpcont} follows 
by using the estimates for $ \I $, $ \II $ and $ \III $.
\end{proof}

\begin{remark}
\label{remarks-after-proof}
There are multiple fine details that could be improved further.
However, 
for the sake of transparent presentation,
we have attempted to keep the formal statement of the theorem simple.
We comment a few of such details below.
\begin{itemize}
  \item The H\"older exponent $\beta$ only affects the estimates locally.
  Suppose that the boundary of $\Omega$ is covered by rectangles $R_i$
  so that $\partial \Omega \cap R_i$ is a graph of a $\beta_i$ H\"older continuous function
  as in the proof. 
  Then we can replace the constant exponent $\beta$ in all the estimates 
  by a variable exponent $\beta : \Omega \to [0,1]$
  defined as 
  \[
  \beta(t,x) = \min_{(t,x) \in R_i} \beta_i
  \]
  if the right hand side is well-defined and as $\beta(t,x) =1$ otherwise.
  In particular, if the domain is smooth except for a single cusp,
  then it is possible to construct $B$ such that estimates for a smooth domain hold except for an arbitrarily small neighborhood of the singularity of the boundary.
  \item The same observation applies to the time regularity parameter $\alpha$.
  Even if $\alpha = 0$, it is possible to obtain estimates for the time derivatives in the interior of the domain. The proof shows that the weight $d(t,x)^{-1/\alpha}$ does not become infinity but depends on the modulus of continuity of the boundary graph.
  \item The constants $C$ in the estimates can be written as $C_1 + C_2$, 
where $C_1$ depends on the domain only through a linear dependency on
\[
\sum_{i=0}^{A} \no{\psi_i}_{\textrm{H\"older}}
\]
where $A$ is the number of boundary patches from Definition \ref{def:domains}
and $\psi_i$ are the corresponding graph functions, 
and through the numerical parameters $\alpha$ and $\beta$.
The constant $C_2$ only depends on the smooth interior domain away from the boundary patches.
This dependency is not the focus of the present paper (see Proposition \ref{prop:trivial_space} and the preceding discussion).
  \item These remarks also apply to Theorem \ref{thm:high}.
\end{itemize}
\end{remark}

The weights on the left hand side quantify a possible boundary blow-up of the gradient.
We can also reformulate it in terms of unweighted $ L^{p} $ scale.
We clarify the interdependency of the weight and 
the integrability gap by recalling the following two propositions.

\begin{proposition}
  \label{prop:tounweighted}
  Assume that $ S \subset \R ^{n} $ is a bounded domain
  such that for some $C_{S} \ge 0$, 
  some $\theta \in (0,1]$ 
  and all $\varepsilon > 0$
  \[
  |\{ x \in \Omega: d(x) < \varepsilon  \}| \le C_S \varepsilon^{\theta}.
  \]
  Let ${{{{\eta}}}} > 0$.
  Then for all test functions $ f $
  \[
 \|f \|_{q} \le C\|f  d ^{{{{{\eta}}}}}\|_p,
  \]
  whenever
  \[
  \frac{{{{{\eta}}}}}{\theta} < \frac{1 }{q } - \frac{1}{p}.
  \]
  The constant $ C $ is independent of $ f $.
\end{proposition}
\begin{proof}
  By H\"older's inequality
  \begin{equation*}
   \int |f(x)|^{q}  \, dx \le 
   \left( \int |f(x)|^{p} d(x)^{p{{{{\eta}}}}}  \, dx \right)^{q/p} 
      \left( \int d(x)^{-{{{{\eta}}}} p q/(p-q) }  \, dx \right)^{1-q/p}.
  \end{equation*}
  By the assumption,
  for all $m \ge 0$
  \[
  |\{(x) \in S: 2^{-m-1} \le d(x) < 2^{-m}  \}| \le C 2^{-m \theta },
  \]
  and we see that 
  \[
  \int d(x)^{ - {{{{\eta}}}}  p q/(p-q)} \, dx
  \le C \sum_{m=1}^{\infty} 2^{-m\theta} \cdot 2^{m {{{{\eta}}}} p q/(p-q)}
  \]
  converges whenever 
  \[
  \frac{{{{{\eta}}}}}{\theta} < \frac{1}{q} - \frac{1}{p}.
  \]
\end{proof}

\begin{proposition}
  \label{prop:examples}
  Let $\Omega_t$ be a bounded $\beta$-H\"older domain.
  Define $\theta$ as in Proposition \ref{prop:tounweighted}.
  \begin{itemize}
    \item It holds $\theta \ge \beta$.
    \item If $\Omega_t$ is Lipschitz except for finitely many power type cusps,
    then $\theta = 1$.
    \item If $\partial \Omega_t$ is $n-1$ rectifiable and $\mathcal{H}^{n-1}(\partial \Omega_t) < \infty$, then $\theta = 1$.
  \end{itemize}
\end{proposition}
\begin{proof}
  Given a graph of a general $\beta$-H\"older continuous function,
we can cover its $\varepsilon$ neighborhood by 
$ \varepsilon^{-(n-1) + (\beta-1) }$ pieces of $n$ dimensional boxes with diameter 
roughly $\varepsilon$.
Hence $\theta \ge \beta$ for general H\"older domains. 
The second item is a special case of the third item.
We prove the third item.
Rectifiability implies that the $n-1$ dimensional Minkowski content of $\partial \Omega_t$
equals its Hausdorff measure,
which we assume to be finite, see 3.2.39 in \cite{MR0257325}.
By the definition of the Minkowski content
and the fact that $\Omega_t$ is bounded,
we conclude there is $C$ such that for all $\varepsilon > 0$
\[
  |\{ x \in \Omega_t: d(t,x) < \varepsilon  \}| \le C \varepsilon.
\]
This concludes the proof.
\end{proof}

Using the propositions,
we can embed the target spaces in Theorem \ref{thm:bogovski_on_w1p}
into unweighted spaces.

\begin{corollary}
  \label{cor:rectifiable}
  Let $\Omega$ be a $C^{0,\beta}$ domain and $\theta$, $p$ and $q$ as above. 
  Then 
  \begin{align*}
   \|\nabla Bf(t,\cdot)\|_q &\le C \|f(t,\cdot)\|_p   
  \end{align*}
  for $(1-\beta)/\theta < 1/q - 1/p$.
  In particular,
  the $\theta = 1$ when $\partial \Omega_t$ is rectifiable with finite area 
  and $\theta \ge \beta$ in the generic case.
\end{corollary}
\begin{proof}
  Theorem \ref{thm:bogovski_on_w1p}, Proposition \ref{prop:examples} and Proposition \ref{prop:tounweighted}.
\end{proof}

\section{Regularity of high order}
%\label{sec:high}
Next we prove estimates involving higher order Sobolev norms.
We separate these estimates from Theorem \ref{thm:bogovski_on_w1p} for better readability,
although they follow by essentially the same arguments.

\begin{theorem}
  \label{thm:high}
Let $\alpha \in [0,1]$ and $ \beta \in (0,1] $. 
Let $ \Omega \subset \R ^{1+n} $ be a $ C^{\alpha,\beta,\theta}$ 
and consider the operator $B$
from Theorem \ref{thm:bogovski_on_w1p}.
Then for all $ p > 1 $, $ k \ge 1 $, $ \kappa \in \{0,1\} $ and $f$ compactly supported
\begin{multline}
  \label{eqthm:highso}
  \left(\sum_{\substack{|\gamma| = k+1 \\ \gamma \in \{0\} \times \N^{n}  } } \int |\partial^{\gamma} \Delta_h^{\kappa} B f(t,x)|^{p} d(t,x)^{(1-\beta)p} \, dx\right)^{\frac{1}{p}}
   \\ \le C  \left( \sum_{\substack{|\gamma| \le k  \\ \gamma \in \{0\} \times \N^{n}  }} \int |\partial^{\gamma} \Delta_h^{\kappa} f(t,x)|^{p}d(t,x)^{ \frac{(|\gamma|-k)p}{\beta}   } \, dx\right)^{\frac{1}{p}}+ \kappa o(1)
\end{multline}
uniformly in $ t $ and asymptotically as $ h \to 0 $.
If $f$ is not compactly supported, 
the estimate still holds with $\kappa = 0$.

Further, if $\alpha > 0$,
we have the time regularity bound
\begin{multline}
  \label{eqthm:highsot}
  \left(\sum_{\substack{|\gamma| = k+1 \\ \gamma \in \{0\} \times \N^{n}  } } 
  \int |\partial^{\gamma}\partial_t^{\kappa} B f(t,x)|^{p} d(t,x)^{(1-\beta)p} \,dx \right)^{\frac{1}{p}}
   \\ \le C  \left(\sum_{\lambda=0}^{\kappa} \sum_{\substack{|\gamma| \le k  \\ \gamma \in \{0\} \times \N^{n}  }} \int |\partial^{\gamma}\partial_t^{\lambda} f(t,x)|^{p}d(t,x)^{ \frac{(|\gamma|-k)p}{\beta}- \frac{(\kappa - \lambda)p}{\alpha \beta}} \, dx \right)^{\frac{1}{p}}
\end{multline}
valid for all integers $ \kappa \ge 0 $, 
uniformly in $t$.
The constant $ C $  depends on the domain, $ n $, $ p $, $ \kappa $ and $ k $.
The norms on the left hand sides of \eqref{eqthm:highso} 
and \eqref{eqthm:highsot}
can be replaced by an unweighted $ \dot{W}^{k+1,q} $ norm
under the conditions as stated in Proposition \ref{prop:tounweighted}.
\end{theorem}

\begin{remark}
The weights in the above estimate can also be located differently.
The proof allows for setting $ d(t,x)^{b} $ to the integral on the left hand side 
and $ d(t,x)^{b/\beta}$ to all terms on right hand side.
Here $ b $ must be non-positive.

To facilitate the later reference, 
we also note that the estimate \eqref{eqthm:highso} can be written as 
\[
 \no{\Delta_h Bf}_{W^{k+1,p,\infty,0,\infty,\beta-1}(t,\Omega)} \le C \no{\Delta_h  f}_{W^{k,p,\beta,0,\infty,0}(t,\Omega)} + o(1)
\]
and the estimate \eqref{eqthm:highsot} reads 
\[
 \no{Bf}_{W^{k+1,p,\infty,\kappa,\infty,\beta-1}(t,\Omega)} \le C \no{f}_{W^{k,p,\beta,\kappa,\alpha \beta,0}(t,\Omega)}
\]
when using the notation from Section \ref{sec:function}.
Note that there are additional terms on the left hand side,
but those only lead to lower order contributions,
which can be controlled by the same right hand side as the highest order contribution.
\end{remark}

\begin{proof}
The proof is a continuation of the proof of Theorem \ref{thm:bogovski_on_w1p}.
We let $ B$ be the operator from Theorem \ref{thm:bogovski_on_w1p},
and consider the notation as given in Subsections \ref{sec:whitney} and \ref{sec:construction}.

We first estimate  
\begin{equation*}
\sum_{|\gamma| = k+1}  \int_{Q_{j_0}^{*}} | \Delta_h^{\kappa}\partial^{\gamma}    Bf  |^{p}\, dx 
\le \sum_{j \in J(k,j_0)}
\sum_{|\gamma| = k+1}  \int_{Q_j^{*}} |\partial^{\gamma} B_{Q_{j}}  \Delta_h^{\kappa} f_{j}  |^{p}\, dx 
\end{equation*}
where $ \kappa \in \{0,1\} $.
Using the bound for $ B_{Q_j} $,
we estimate a single $ j $ term by
\begin{equation*}
  \sum_{|\gamma| = k}  \int |\partial^{\gamma}  \Delta_h^{\kappa}f_{j}  |^{p}\, dx 
  = \sum_{|\gamma| = k}   \int |\partial^{\gamma} \Delta_h^{\kappa}  (\nabla \chi_j \cdot B' f ) 
  +\partial^{\gamma}   \Delta_h^{\kappa}  ( \chi_jf )|^{p} \, dx.
\end{equation*}
By Leibniz rule and the derivative bounds for $ \chi_j $ based on the Whitney property,
the second term is of the desired form.

We focus on the first term.
Using the Leibniz rule,
we find an upper bound by
\begin{equation*}
  C \sum_{\lambda=0}^{\kappa} \sum_{l=0}^{k} \sum_{ \substack{|\gamma| = l}} r_j^{- ( 1+ k - l)p}(h /\delta(r_j))^{\kappa - \lambda}  \int_{Q_j^{*}} |\partial^{\gamma} \Delta_h^{\lambda}B' f  |^{p}\, dx ,
\end{equation*}
where we used that 
\[
|\Delta_h^{\lambda}\partial^{\gamma'}\nabla\chi_j| \lesssim r_j^{-( 1 + k - l)}(h /\delta(r_j))^{\kappa - \lambda}
\]
for $ |\gamma'| = k-l $ and $ 0 \le \lambda \le \kappa $.
Next we notice that the term with $ \kappa - \lambda = 1 $
is of the order $ o(1) $ as $ h \to 0 $.

Collecting the argument together,
we bound
\[
\sum_{j_0} \sum_{|\gamma| = k+1}  r_{j_0}^{(1-\beta)p} \int_{Q_{j_0}^{*}}  |\partial^{\gamma} \Delta_h^{\kappa}    Bf  |^{p}\, dx \le \I + \II
\]
where 
\begin{align}
  \label{eq:alternativeforhardy1}
\I &= C \sum_{\substack{|\gamma| \le k  \\ \gamma \in \{0\} \times \N^{n}  }} \int |\partial^{\gamma} \Delta_h^{\kappa} f(t,x)|^{p}d(t,x)^{(|\gamma|-k)p} \, dx \\
\II &=  C\sum_{l \le k} \sum_{ \substack{|\gamma| = l}}  \int  | \Delta_h^{\kappa}\partial^{\gamma} B' f  |^{p}d(t,x)^{- (k - l + \beta)p} \, dx + \kappa o(1). \nonumber
\end{align}
We use the definition \eqref{eq:bprime} and Propositions \ref{prop:interior_bogovski}
and \ref{prop:bdry_bogovski} to estimate 
\begin{equation}
  \label{eq:alternativeforhardy2}
\II \le C \sum_{l \le k} \sum_{ \substack{|\gamma| = l}}  \int  | B' \Delta_h^{\kappa}\partial^{\gamma} f  |^{p}d(t,x)^{ -(k - l + \beta)p} \, dx + \kappa o(1).
\end{equation} 
Then, for $h$ small enough, 
we can apply Proposition \ref{prop:bdry_bogovski} to conclude \eqref{eqthm:highso}.

The case with time derivatives is not much different.
We can replace the finite difference $\Delta_h$
by an iterated differentiation in time
and replace the use of product rule for finite differences 
by the Leibniz rule for iterated differentiation.
The estimate \eqref{eqthm:highsot} then follows as above.
\end{proof}

\begin{corollary}
  \label{cor:highso_lip}
  If $\Omega$ is $C^{\alpha,\beta,\theta}$ with $0 < \alpha, \beta \le 1 $
  and $\Omega_t$ satisfies Hardy's inequality \eqref{eq:hardy_lipschitz} for negative powers,
  then the right hand side of \eqref{eqthm:highsot} can be replaced by 
  \[
   C  \left(\sum_{\lambda=0}^{\kappa} \sum_{\substack{|\gamma| = k  \\ \gamma \in \{0\} \times \N^{n}  }} \int |\partial^{\gamma}\partial_t^{\lambda} f(t,x)|^{p}d(t,x)^{\frac{(\kappa - \lambda)p}{\alpha \beta}} \, dx \right)^{\frac{1}{p}}
  \]
  and the norm on the right hand side of \eqref{eqthm:highso} can be replaced by $C \no{f}_{\dot{W}^{k,p}(\Omega_t)}$.
  Here $f$ is assumed to be compactly supported.
\end{corollary}
\begin{proof}
  The left hand side of \eqref{eqthm:highsot}
  is bounded by $\I$ and $\II$ in \eqref{eq:alternativeforhardy1} and \eqref{eq:alternativeforhardy2}
  with $\Delta_h$ replaced by $\partial_t$ and $o(1)$ replaced by zero.
  Applying Hardy's inequality \eqref{eq:hardy_lipschitz} 
  repeatedly and then the macroscopic estimates for definition \eqref{eq:bprime}, Proposition \ref{prop:interior_bogovski}
  and derivatives bounds in Proposition \ref{prop:bdry_bogovski}, 
  we estimate the version of $\II$ with time derivatives by 
  \[
     C \sum_{\lambda = 0}^{\kappa} \sum_{l \le k} \sum_{ \substack{|\gamma| = l}}  \int  |  B' \partial_t^{\kappa- \lambda }\partial^{\gamma} f  |^{p} d(t,x)^{- \beta p + (\lambda- \kappa)p/\alpha} \, dx  .
  \]
  Applying then the weighted $L^{p}$ bound in Proposition \ref{prop:bdry_bogovski},
  we conclude the proof for $\II$.
  The estimate for $\I$ is simpler as we can skip 
  the part of the argument related to $B'$ 
  but otherwise argue as above with Hardy's inequality.
\end{proof}

\begin{remark}
  
  If $\beta = 1$, the assumptions of the Corollary \ref{cor:highso_lip}
  are satisfied.
  In addition, any $C^{\alpha,\beta,\theta}$ domain in $\R^{1+2}$
  satisfies the assumptions as discussed in connection with Theorem 1.2 in \cite{MR3168477}.
\end{remark}

Next we discuss the negative order estimates.
These estimates follow from the positive order estimates for the adjoint operator.
We point out that the unweighted version of 
the theorem below holds for the classical construction in Lipschitz domains
with zero boundary values  when $k <  2 - 1/p $.
For lower values of $k$ it holds with zero boundary values on the right hand side but not on the left hand side.
In accordance,
we obtain an unweighted bound with zero boundary values for $ k = 1$ and $\beta = 1$ on the left hand side, 
but the value $k = 2$ already creates a weight function to quantify 
the possible failure of the unweighted estimates. 
The claim could not be proved with input in the dual of inhomogeneous Sobolev space,
but the dual of homogeneous Sobolev space would already do in the stationary Lipschitz case. 
We refer to the discussion starting on page 182 of the book \cite{MR2808162} as well as the paper \cite{MR3168477} for more,
but we do not attempt to prove results beyond zero boundary value input in our setting.

\begin{theorem}
\label{thm:negat}
Let $ \Omega \subset \R ^{1+n} $ be a $ C^{\alpha,\beta,\theta} $ domain with $0 < \alpha \le 1$
and $0 <\beta \le 1$.
Fix an integer $k > 0$, 
a real number $1<p<\infty$, 
and another integer $\lambda \ge 0$.
If $\beta > 0$, 
then for any $ f \in C_{smz}^{\infty} $ and 
any time $t$ %and $\kappa \in \{0,1\}$
\begin{align}
\label{eqthm:negattime1}
\|\partial_t^{\lambda} Bf\|_{W_0^{-k+1,p,-1,0,\infty, \beta -1 }(t,\Omega)} &\le C  \| f\|_{W_0^{-k,p,\infty,\lambda, \alpha \beta,0 } (\Omega_t)}  .
\end{align}
% \begin{align}
% \label{eqthm:negattime1}
% \|\Delta_h^{\kappa} \partial_t^{\lambda} Bf\|_{W_0^{-k+1,p,-1,0,\infty, \beta -1 }(t,\Omega)} &\le C  \|\Delta_h^{\kappa} f\|_{W_0^{-k,p,\infty,\lambda, \alpha \beta,0 } (\Omega_t)}  + \kappa o(1)
% \end{align}
% holds asymptotically as $h \to 0$.
% If $\kappa = 0$,
% the assumption on compact support of $f$ can be removed. 
If $\lambda = 0$,
the bound holds with $\beta = 0$.
\end{theorem}

\begin{proof}
To carry out the proof, 
we need to study the adjoint of $B$ in detail. In principle we need to dualize the {\em local Bogovskij operators} on the Whitney-type cover first, then the decomposition operator and finally the first auxiliary operator.
We prove the inequality \eqref{eqthm:negattime1} first.
Using the notation from Sections \ref{sec:auxop}, \ref{sec:whitney}, \ref{sec:construction} and Remark~\ref{rmk:dualbprime}
we write 
\begin{equation*}
%\label{eq:sumdecomp}
\partial_t^{\lambda}Bf = \sum_{\lambda' = 0}^{\lambda} \binom{\lambda}{\lambda'}  \sum_{j} B_{Q_j}\left( \partial_t^{\lambda'}\nabla \chi_j \cdot \partial_t^{\lambda- \lambda'}B' f + \partial_t^{\lambda'} \chi_j \partial_t^{\lambda-\lambda'}f  \right) .
\end{equation*}
As before we divide $B'$ defined in \eqref{eq:bprime}
into two parts
\[
B_1 = \sum_{i} \varphi_i B_{ext,i}(\eta_i f) , \quad B_2 =  B_{int}(f- \sum_{i} \varphi_i \eta_i f - \sum_{i} \nabla \varphi_i \cdot B_{ext,i}(\eta_i f) ).
\]
Let $\zeta_1$ be a smooth cut-off that equals one where 
\[
\min_i \eta_i > 0
\]
and let $\zeta_2$ be a smooth cut-off that is equal to one
where 
\[
f- \sum_{i} \varphi_i \eta_i f - \sum_{i} \nabla \varphi_i \cdot B_{ext,i}(\eta_i f) \ne 0.
\]
We define the relevant (almost) adjoint operators 
\begin{equation*}
%\label{eq:bprimeadjoint1}
  B_{1,\lambda'}^{*} g =  \zeta_1  \sum_{i}  (\partial_t^{\lambda' } \eta_i)  B_{ext,i}^{*}( \varphi_i  g  )
\end{equation*}
and 
\begin{multline*}
B_{2,\lambda'}^{*} g =  
\zeta_2   B_{int,\lambda'}^{*}g \\
 - \zeta_2 \sum_{\lambda'' = 0}^{\lambda'} \sum_{i}  \binom{\lambda'}{\lambda''}    
(\partial_t^{\lambda' - \lambda''}\eta_i) (  \varphi_i  B_{int,\lambda''}^{*}g - \sum_{i}   B_{ext,i}^{*}(B_{int, \lambda''}^{*}g \nabla    \varphi_i)  ). 
\end{multline*} 
Here $B_{int,\lambda'}^{*}$ is the same operator as in \eqref{decopsadjoint}.

By Proposition \ref{prop:adjointtrivial}, 
we find 
\begin{multline}
\label{adjoint2}
 \int_{\Omega_t}   |\partial^{\gamma} B_{1,\lambda'}^{*}g(t,x)|  ^{p} d(t,x)^{b\beta}  \, dx \\
    \le C   \sum_{\substack{\gamma' \in \{0\} \times \N^{n} \\ |\gamma'| \le |\gamma|}} \int_{\Omega_t} | \partial^{\gamma}g(t,x) (t,x)|^{p}d(t,x)^{ \Theta b+ \beta p } \, dx 
\end{multline}
for all $b > -1 $
where 
\[
\Theta = \begin{cases}
  \beta, \quad 0 \ge b > -1 \\
 \beta^{2}, \quad b > 0
\end{cases}.
\]
On the other hand,
$  d(t,x) \sim 1$ in the support of $\zeta_2$,
and so we can use Proposition \ref{prop:interior_bogovski}
and Poincar\'e's inequality to conclude that 
\eqref{adjoint2} holds with $B_{1,\lambda'}^{*}$ replaced by $B_{2,\lambda'}^{*}$,
as shown in \eqref{decopsadjoint}.
Set
\begin{equation}
\label{negatproof:decomposition}
B_{\lambda'}^{*} g = \sum_{j} \left( (B_{1,\lambda'}^{*} + B_{2,\lambda'}^{*}) (B_{Q_j}^{*}(g\tilde{\chi}_j) \nabla \partial_t^{\lambda'} \chi_j)  + (\partial_t^{\lambda'}\chi_j)(B_{Q_j}^{*}(g\tilde{\chi}_j))\right) 
\end{equation}
where $\tilde{\chi}_j$ is a bump function equal to one in the support of $B_{Q_j}f_j$ (with finite overlap and the usual bounds).

To prove the bound in the statement of the theorem,
take a smooth test function $\psi$.
We may assume without loss of generality that $\partial_t^{\lambda} Bf$ is compactly supported.
Now 
\begin{equation*}
|\la{ \partial_t^{\lambda} Bf, \psi}| 
  = \bigg|\sum_{\lambda' = 0}^{\lambda} \binom{\lambda}{\lambda'}  \la{\partial_t^{\lambda - \lambda'}f, B_{\lambda'}^{*}\psi }\bigg|
  \le \sum_{\lambda' = 0}^{\lambda} \binom{\lambda}{\lambda'}  
   |\la{   \partial_t^{\lambda - \lambda'}f, B_{\lambda'}^{*}\psi }| .
\end{equation*}
Hence it suffices to show that 
\begin{equation}
\label{eq:innegatthm}
\no{ B_{\lambda'}^{*} \psi}_{W^{k,p,\infty,0,\infty,- \lambda'/(\alpha \beta)}(t,\Omega)} \le  
 C \no{ \psi }_{W^{k-1,p,1,0,\infty, 1-\beta }(t,\Omega)}
\end{equation}
and we may estimate the terms $B_{1,\lambda'}^{*}$ and $B_{2,\lambda'}^{*}$ as well as the third term in the definition separately.
Note that as the starred operators never have compact support,
the right hand side of the norm inequality \eqref{eqthm:negattime1} is bound to have a function space with subscript zero
as we insist on having $k\ge 1$.

We start the estimation from $B_{1,\lambda'}^{*}$. The estimates for the other terms follow by the same argument.
It follows from \eqref{adjoint2} with $b = p \lambda'/ (\alpha \beta) $, 
the Poincar\'e inequality 
and Proposition \ref{prop:reference}
that 
\begin{align}
\label{eq:bstarest}
\begin{aligned}
&\no{ B_{1,\lambda'}^{*} \psi}_{W^{k,p,\infty,0,\infty,-\lambda'/(\alpha \beta)}(t,\Omega)}
 \\
&\quad  \le \sum_{|\gamma|  \le  k}  \int | \partial^{\gamma} 
B_{1,\lambda'}^{*}        \left( \sum_j B_{Q_j}^{*}(\psi \tilde{\chi}_j) \nabla \partial_t^{\lambda'} \chi_j \right)       |^{p} d(t,x)^{p \lambda'/(\alpha \beta ) } \, d x \\
&\quad   \le C \sum_{j} \sum_{|\gamma| \le k} r_{j}^{(| \gamma|-k)p +(\beta-1)p } \int_{P_j^{*}}   | \partial^{\gamma} B_{Q_j}^{*}(\tilde{\chi}_j \psi) |^{p}  \, d x \\
&\quad  \le C \sum_{j} \sum_{|\gamma|+ 1 \le k} r_{j}^{(| \gamma| +1 -k)p + (\beta -1)p} \int_{\Omega_t} | \partial^{\gamma} (\tilde{\chi}_j \psi) |^{p}  \, dx \\
  &\quad   \stackrel{*}{\le}  C \sum_{|\gamma| \le k-1} \int_{\Omega_t} |\partial^{\gamma} \psi(t,x)|^{p}d(t,x)^{(| \gamma|-k +1)p  + (\beta - 1)p} \, dx.
\end{aligned}
\end{align} 
This is the desired bound for $B_{1,\lambda'}^{*}$.
As the bound for $B_{2,\lambda'}^{*}$ and the bound for the third term in \eqref{negatproof:decomposition} 
follow by the same computation,
the proof of \eqref{eq:innegatthm} is complete.
The claim \eqref{eqthm:negattime1} follows by the definition of the norms.
\end{proof}

\begin{remark}
One time derivative can be replaced by a 
finite difference as was done before using that the input function $f$ is compactly supported.
We omit the details of the proof.
%as this side remark has no other consequences than Corollary \ref{cor:timecont1}.
\end{remark}

Whereas Theorem \ref{thm:negat} dealt with duals of intersection spaces,
we can reformulate a result for duals of sum spaces.
In the proposition below,
an intersection space gets mapped into a sum space,
which looks like a weaker result,
but from the point of view of applications using an input data in an intersection space is very natural.

\begin{proposition}
\label{prop:new}
Assume that $\Omega\in C^{\alpha,\beta,\theta}$ with $0 < \alpha,\beta \le 1$.
Let $\lambda \ge 0$ be an integer,
and for each $0 \le \lambda' \le \lambda$,
take $k_{\lambda'} \ge 1$ and $p_{\lambda'} \in (1,\infty) $.
Then for all $f \in C_{smz}^{\infty}(\Omega)$
\[
\sup_{\psi} | \la{ \partial_{t}^{\lambda} Bf, \psi } | \le C \sum_{\lambda'=0}^{\lambda} \no{\partial_t^{\lambda-\lambda'}f}_{X_{\lambda'}^{*}}
\]
where the supremum is over all bounded $\psi \in C^{\infty}(\Omega_t)$ with 
\[
\sum_{\lambda'=0}^{\lambda} \no{\psi}_{Y_{\lambda'}} \le 1
\]
and $(X_{\lambda'},Y_{\lambda'})$ is either 
\[
( W^{k_{\lambda'},p_{\lambda'},\infty,0,\infty,0}(t,\Omega), W^{k_{\lambda'}-1,p_{\lambda'},1,0,\infty,1-\beta +  \lambda' /\alpha}(t,\Omega))
\]
or 
\[
( W^{k_{\lambda'}-1,p_{\lambda'},\infty,0,\infty,0}(t,\Omega), W^{k_{\lambda'}-1,p_{\lambda'},1,0,\infty,-\beta +  \lambda' /\alpha}(t,\Omega)).
\]
\end{proposition}

\begin{proof}
We claim $\partial_t^{\lambda} Bf$ to be in 
\[
 \left(\bigcap_{\lambda'= 0}^{\lambda} Y_{\lambda'}  \right)^{*} = \sum_{\lambda'=0}^{\lambda} Y_{\lambda'}^{*}.
\]
We refer back to the proof of Theorem \ref{thm:negat}.
Recall that 
\[
\la{\partial_t^{\lambda} Bf, \psi} 
  = \sum_{\lambda' = 0}^{\lambda} \binom{\lambda}{\lambda'} \la{\partial_t^{\lambda - \lambda'} f  , B_{\lambda'}^{*} \psi}
\]
where we define $B_{\lambda'}^{*}$ according to \eqref{negatproof:decomposition}.
Hence we are left with showing that 
\[
(B_{\lambda'}^{*})^{*} \partial_t^{\lambda - \lambda'} f \in Y_{\lambda'}^{*}.
\]
This follows by estimating $\no{B_{\lambda'}^{*} \psi}_{Y_{\lambda'}}\leq c\no{\psi}_{X_{\lambda}'}$.
%with the relevant norm estimate.
These estimates follow by adjusting the estimates leading to to verification of \eqref{eq:innegatthm} accordingly and we omit the details here. We just mention, that in the case {\em the differentiability is not increased} (which is the second possible choice of pairing) one can decrease the power of the distance using Poincar\'e's inequality. It happens precisely in the inequality marked with an asterisk in \eqref{eq:bstarest}.
\end{proof}

The strength of this formulation lies in the fact 
that the spaces above can often be identified with weighted $L^{p}$ spaces.
As an important special case,
we consider $1<p,s< \infty$ and assume
\[
\partial_t f  \in W^{-1,p}(\Omega_t) , \quad f \in L^{s}(\Omega_t).
\]
Then $\partial_{t} Bf$ is in the dual space of 
\[
d^{(\beta-1)p' }L^{p}(\Omega_t) \cap d^{(\beta   -1/\alpha)s'  }L^{s'}(\Omega_t)
\]
which is
\[
d^{(1-\beta)p }L^{p}(\Omega_t)  +  d^{( -\beta + 1 /\alpha) s  }L^{s}(\Omega_t).
\]
In particular, Proposition \ref{prop_last_intro} follows with aid of this estimate 
and Proposition \ref{prop:tounweighted}.

\begin{corollary}
\label{cor:timecont1}
Let $ \Omega $ be a $ C^{\alpha,\beta,\theta} $ domain with $\alpha , \beta \ge 0$.
Let $ \kappa \ge 0 $ and $p \in (1,\infty)$.
Let $f$ be a limit of functions in $ C_{0,smz}^{\infty}(\Omega)$
and assume, asymptotically as $h \to 0$,
\[ \sup_{t} \no{ \Delta_h f}_{W_0^{-1,p,\infty,\kappa,\alpha \beta,0}(t,\Omega)} = o(1) \] 
if $ \alpha , \beta > 0$ and 
\[ \sup_{t} \no{ \Delta_h f}_{W_0^{-1,p,\infty,0,\infty,0}(t,\Omega)} = o(1) \] 
if $\alpha = \beta = \kappa = 0$.

Then $  \partial_t^{\lambda} Bf \in C(0,T; L^{q}(\rn)) $ for all $0 \le \lambda \le \kappa$
and $q > 0$ such that 
\[
\frac{(1+ \varepsilon)(1-\beta)}{\theta} \le   \frac{1}{q} - \frac{1}{p}
\]
for some $\varepsilon > 0$.
\end{corollary}

\begin{proof}
  Take $ t $.
  We handle the cases $\beta = 0$ and $\beta > 0$ simultaneously 
  using the convention $0/0 = 0$ for the quantity $(\lambda - \kappa)/(\alpha \beta)$
  appearing in the definition of the norm.
  We want to show
  \[ \no{\Delta_h \partial_t^{\kappa} Bf }_{q} \longrightarrow 0, \quad h \longrightarrow 0 .\]
  By Proposition \ref{prop:tounweighted} $\no{\Delta_h \partial_t^{\kappa} Bf }_{q} \le C \no{d^{1-\beta} \Delta_h \partial_t^{\kappa} Bf }_{p} $.
  Let $ \epsilon > 0 $.
  By the a priori bounds,
  we can find a test function $ f'$
  such that for all $0 \le \lambda \le \kappa$
  \[
  \sup_{t} \left( \no{ f - f' }_{W^{-1,p,-1,\kappa,\alpha \beta,0}(t,\Omega)} 
  + \no{ \partial_t^{\lambda} Bf -\partial_t^{\lambda} Bf'}_{L^{q}} \right)
  < \epsilon,
  \]
  the first term coming from the definition of the space,
  the second coming from the definition of $\partial_t^{\lambda} B $ as extension from a dense subspace. 
  We also used the fact  
\[
W_0^{0,p,-1,\kappa,\infty,\beta - 1}(t,\Omega) = W_0^{0,p,-\infty,\kappa,\infty,\beta - 1}(t,\Omega) .
\]
 
We use \eqref{eqthm:negattime1} an the remark after Theorem \ref{thm:negat} to compute
\begin{multline*}
  \no{\Delta_h \partial_t^{\kappa} Bf}_q
  \le
  \no{\Delta_h Bf'}_q +
  \no{\Delta_h \partial_t^{\kappa} (Bf-Bf')}_q \\
  < C   \no{\Delta_h f'}_{W_0^{-1,p,-1,\kappa,\alpha \beta,0}(t,\Omega)} + 2\epsilon.
\end{multline*}
Here the norm on the right hand side is bounded by
\[
 \no{\Delta_h f}_{W_0^{-1,p,\infty,\kappa,\alpha \beta,0}(t,\Omega)} + 
    \no{\Delta_h (f-f')}_{W_0^{-1,p,\infty,\kappa,\alpha \beta,0}(t,\Omega)} 
\le o(1) + 2\epsilon
\]
Choosing $h$ small enough, 
we see that $\no{\Delta_h \partial_t^{\kappa} Bf}_q < \epsilon$ and the proof is complete.
\end{proof}

\begin{corollary}
Let $\beta \in (0,1] $, $ k \ge 1 $ and 
$ 1<p<\infty $.
Let $ \Omega $ be a $ C^{0,\beta} $ domain.

If $f$ is a limit of functions in $ C_{0,smz}^{\infty}(\Omega)$ and 
\[ \sup_{t} \no{ \Delta_h f}_{W^{k,p,\beta,0,\infty,0}(t,\Omega)} = o(1), \] 
then 
\[
 Bf \in C(0,T; W^{k+1,q}(\rn))
\]  
for all $q > 0$ such that 
\[
\frac{(1+\varepsilon)(1- \beta)}{\theta}  \le  \frac{1}{q} - \frac{1}{p}
\]
for some $\varepsilon > 0$.
 
If in addition $ \beta = 1 $ and if
$f$ is a limit of $C_{0,smz}^{\infty}(\Omega)$ functions
with $f \in C(0,T;\dot{W}^{k,p}(\rn))$,
then $ Bf \in C(0,T; W^{k+1,p}(\rn)) $.
\end{corollary}
\begin{proof}
The argument is essentially the same as for Corollary \ref{cor:timecont1}. 
\end{proof}

\begin{corollary}
Let $ \Omega $ be a $ C^{\alpha,\beta,\theta} $ domain and assume that $ 1<p<\infty $.
If $f$ is a limit of functions in $ C_{smz}^{\infty}(\Omega)$,
\[
 \partial_t f,\ d^{-1/(\alpha \beta)} f  \in  L^{p}(\Omega).
\]
Then 
\[
  \nabla Bf \in C^{1-1/p}(0,T;L^{q}(\rn))
\] 
for all $q > 0$ such that 
\[
\frac{(1+\varepsilon)(1- \beta)}{\theta}  \le \frac{1}{q} - \frac{1}{p}
\]
for some $\varepsilon > 0$.
\end{corollary}

\section{Pressure estimates by duality}
\label{sec:pressureduality}
In the remaining sections, 
we demonstrate how the inverse of the divergence can be used to provide pressure estimates
for the Navier--Stokes system from Subsection \ref{sec:introfluid}.
We study weak and even very weak solutions.
The gradient of the pressure is then merely a distribution,
and its regularity is expressed most naturally in terms of duality. 
The time independent version of the relevant duality statement 
is known as Lions--Ne\v{c}as negative norm theorem.
In the following subsection, 
we produce an analogous statement in the context of pressure estimates.

\subsection{Negative norm theorem}
The choice of the set of test functions in the weak formulation of the equation is important here.
While Definition~\ref{def:fluid_intro} 
only asks the equation to be valid for compactly supported test-functions, 
the set of test-functions 
can be much larger for instance
in the case of outflow boundary conditions or in the case of fluid-structure interactions. 

As will be seen below, 
the pressure can always be reconstructed up to a function only depending on time. 
We will construct the pressure in the dual space of functions for which 
\[ \int_{\partial\Omega_t} \psi(t,x) \cdot \nu(t,x)\, d\mathcal{H}^{n-1}(x)=0 \] 
holds in a suitable weak sense.
Here $\nu$ is the outer normal of the time-slice $\Omega_t$.
It need not be defined in any classical sense,
but we can define the class of test functions with the property above 
by 
\begin{align*}
C^\infty_{\pi^*}(\Omega; \rn )= \left \lbrace \psi\in C^\infty(\R^{1+n}; \rn ): \int_{\Omega_t} \dive  \psi(t,x) \, dx = 0, \ t\in (0,T) \right \rbrace
\end{align*}
and use the subscript $\pi^{*}$ in accordance with the conventions in Subsection \ref{sec:function}
to define the spaces 
\[
L^{a}W_{\pi^{*}}^{\aleph}
\]
as the closure of $C^\infty_{\pi^*}(\Omega; \rn )$ with respect to the relevant norms.

This set-up allows us to reconstruct the pressure up to its mean value in space.
It is noteworthy
that the mean value of the pressure is often an invariant of a solution,
which has to be given as an additional information.
We can include the full family of impermeable boundary conditions, 
where $v(t,x)\cdot \nu(t,x)=0$ whenever $x \in\partial\Omega_t$.
Accordingly we generalize Definition~\ref{def:fluid_intro} by assuming that 
\begin{center}
  $\la{\Lambda ,\psi} = \la{s_0 ,\psi}$ for divergence free test functions
\end{center}
where $\Lambda$ is the sum of the distributions appearing on the left hand side 
of the equation
and $s_0 \in C^{\infty}(\Omega;\rn)^{*}$ is a functional with $\la{s_0 ,\psi}=0$ if 
\[
\psi\in C^\infty (\Omega; \rn ) \cap W_{0}^{1,1}(\Omega;\rn).
\]
We then aim to solve
\[
\nabla \pi=\Lambda
\]
for $\pi$.

The statement below can be seen as a generalized negative norm theorem.
The reader is advised to remember the notation for $L^{a}W^{\aleph}$ from 
Section \ref{sec:function} before reading further.
\begin{lemma}
  \label{lem:distribution}
Let $\alpha, \beta \in [0,1]$ and $\Omega$ be a $ C^{\alpha,\beta,\theta}$ domain.
Consider a vector-space of test functions $A$ with $C_{0}^{\infty}(\Omega; \rn )\subset A \subset C^{\infty}_{\pi^*}(\Omega; \rn )$. 
In case $C_{0}^{\infty}(\Omega; \rn ) \ne A$,
consider an additional functional $ s_0 \in A^*\subset C_{0}^{\infty}(\Omega; \rn )^{*}$ 
with $\la{s_0,\psi}=0$ for all $\psi\in C_{0}^{\infty}(\Omega; \rn )$.

If $ \Lambda_i \in A^*$ are such that 
\[
 \la{\Lambda ,\psi} = \la{s_0 ,\psi}, \quad \sum_{i=1}^{m} \Lambda_i = \Lambda
\]
for all divergence free test functions $ \psi \in A $,
then there exist $\pi \in C^{\infty}(\Omega)^{*}$
such that for all $ \psi \in A \subset C^{\infty}_{\pi^*}(\Omega;\rn) $
\[
\la{ \pi, \dive  \psi} =  \la{s_0 ,\psi} - \la{\Lambda, \psi },
\]
that is, $ \Lambda_i = \nabla \pi_i $ in the weak sense and $\pi = \sum_{i=1}^{m} \pi_i$.

Given $i$,
let $\kappa \ge 0$, $1< q \le \infty$, $1 < p < \infty$.
Let $\theta \in (0,1]$ 
be such that the assumptions of Proposition \ref{prop:tounweighted} hold.
Assume that $\Lambda_i = \Lambda_i' \circ \partial_t^{\kappa}$.
Then we can arrange the construction to make the following four estimates valid.
\begin{enumerate}
  \item Let $b \in \R$.
Define  
\[
\Theta = \begin{cases}
  1, & \quad  \trm{if $ b \le 0$}\\
  \beta, & \quad \trm{if $0 < b  \le 1$ and $ p < (p-1) /\beta$}\\
  \beta^{2}, & \quad  \trm{if $1 <  b < (p-1)/\beta$ or $0<b <(p-1)/\beta \le p$}
\end{cases}.
\]
If any of the three cases applies to $b$, then 
  \[
  \no{\pi_i}_{L^{q}W_0^{0,p,\infty,-\kappa,- \alpha\beta, \Theta  b}(\Omega)} \le C \no{\Lambda_i'}_{L^{q}W^{0,p,\infty,0,\infty,\beta(  b-1)} (\Omega)}.
  \] 
  \item Let $ k \ge 1 $.
  Then 
  \[
  \no{\pi_i}_{L^{q}W_0^{-k+1,p,-\beta,-\kappa,-\alpha\beta,0}(\Omega)}
  \le C \no{\Lambda_i'}_{L^{q}W^{-k,p,\infty,0,\infty,1-\beta}(\Omega)}.
  \] 
  \item Let $  k \le 0 $ and $ b \ge 0 $. Then
  \begin{align*}
    \no{\pi_i}_{L^{q}W^{-k+1,p,\infty,-\kappa,\infty,-b}(\Omega )} 
&\le C \no{\Lambda_i'}_{ L^{q}W^{-k,p,1,0,\infty,\kappa/\alpha +1 - b\beta }(\Omega)},\\
    \no{\pi_i}_{L^{q}W^{-k+1,p,\infty,-\kappa,-\alpha \beta,0}(\Omega)} 
&\le C \no{\Lambda_i'}_{ L^{q}W^{-k,p,1,0,\infty, 1-\beta }(\Omega)}. 
  \end{align*}  
  \item Let $k \le 0$. Then we can write
  \[
\pi_i = \sum_{\lambda=0}^{\kappa} \pi_i^{\lambda}
\]
with 
\[
\no{  \pi_i^{\kappa - \lambda} \circ \partial_t^{\lambda - \kappa} }_{X_{\lambda}}
  \le C  \no{ \Lambda_i'  }_{Y_{\lambda}}
\]
where each $(X_{\lambda}, Y_{\lambda})$ can be set freely to either
\[
(L^{q_{\lambda}}W^{-k_{\lambda},p_{\lambda},\infty,0,\infty,0}(\Omega), 
L^{q_{\lambda}}W^{-k_{\lambda},p_{\lambda},\infty,0,\infty,-\beta+\lambda/\alpha}(\Omega))
\]
or
\[
(L^{q_{\lambda}}W^{-k_{\lambda}+1,p_{\lambda},\infty,0,\infty,0}(\Omega) , L^{q_{\lambda}}W^{-k_{\lambda},p_{\lambda},\infty,0,\infty,1- \beta +\lambda/\alpha }(\Omega)).
\]
with $k_{\lambda} \le 0$ and $1< p_{\lambda} < \infty$.
\end{enumerate}
The constants $C$ only depend on the domain and the parameters of the function spaces.
% If $\beta = 1$ and $C_{0}^{\infty}(\Omega; \rn ) = A$, 
% then the third index in the superscript of $W$ can be replaced by infinity.
\end{lemma}
\begin{proof}
We study test functions $ \varphi \in C_{\pi^{*}}^{\infty}(\Omega;\rn) $
with 
\[
\sup_{t} \abs{\int \dive\varphi(t,x) \, dx}  = 0.
\]
Then $ \varphi - B \dive \varphi $ is divergence free 
and $ B \dive \varphi $ has zero boundary values.
We define 
\begin{equation*}
\la{\pi_i, \psi} :=  -\la{\Lambda_i, B \psi}, \quad
\Lambda := \sum_{i=1}^{m} \Lambda_i ,
\end{equation*}
for all test functions $\psi \in C^{\infty}(\Omega)$ with 
\[
\sup_{t} \abs{\int_{\Omega_t} \psi(t,x) \, dx} = 0,
\]
and consequently every $ \varphi \in C_{\pi^{*}}^{\infty}(\Omega;\rn) $ satisfies
\begin{multline*}
  \sum_{i=1}^{m} \la{ \pi_i, \dive \varphi}
  =  -\la{\Lambda, B \dive \varphi} 
  =  -\la{\Lambda, \varphi} + \la{\Lambda, \varphi - B \dive \varphi } \\
  =  -\la{\Lambda, \varphi} + \la{s_0 ,\varphi - B \dive \varphi }
  =  -\la{\Lambda, \varphi} + \la{s_0 ,\varphi}
\end{multline*}
where the last equality used that $ B \dive \varphi $ has zero boundary values.

By Theorem \ref{thm:high},
restricting the attention to functions with mean zero 
on all time slices,
$ \partial_t^\kappa B $ is a bounded linear operator
\begin{align*}
L^{q}W^{l,p,\beta,\kappa,\alpha\beta,0} &\longrightarrow
L^{q}W_0^{l+1,p,\infty,0,\infty,\beta-1}, \quad l \ge 0 \\
L^{q}W^{0,p,\infty,\kappa,\alpha \beta,-b} &\longrightarrow
L^{q}W_0^{0,p,\infty,0,\infty,\beta(1-b) }, \quad b \le 0
\end{align*}
and by Theorem \ref{thm:negat} $\partial_t^{\kappa} B$ is bounded
\begin{align*}
L^{q}W_0^{k,p,\infty,\kappa,\infty,b}
&\longrightarrow L^{q}W_{0}^{k+1,p,-1,0,\infty,-\lambda/\alpha- 1 + \beta b}, \quad   k < 0. \\
L^{q}W_0^{k,p,\infty,\kappa,\alpha \beta,0}
&\longrightarrow L^{q}W_{0}^{k+1,p,-1,0,\infty,\beta-1}, \quad   k < 0.
\end{align*}
Writing, $ \pi_i = - \Lambda_i \circ B $,
we conclude that $\pi$
satisfies the claimed bounds in (1), (2) and (3).

To prove (4),
we use the proof of Proposition \ref{prop:new}.
As $\pi_i = - \Lambda_i' \circ \partial_t^{\kappa} \circ B$,
we can use the decomposition from the short proof of Proposition \ref{prop:new}
to define 
\[
\pi_i^{\kappa- \lambda} = \Lambda_i' \circ (B_{\lambda}^{*})^{*} \circ \partial_t^{\kappa - \lambda}.
\]
Then the norm estimates claimed here are dual to the ones shown in the proof of Proposition \ref{prop:new}.
Hence we have completed the proof of (4).

Finally,
by the Hahn--Banach theorem one may extend every $\pi_i$ to act on all functions from $C^{\infty}(\Omega)$ satisfying the relevant bounds.
% Finally, if the additional conditions on $A$ and $\beta$ hold,
% we are in a position to invoke Corollary \ref{cor:highso_lip}
% to conclude the additional statement on the third upper index.
\end{proof}

\begin{remark}
  The construction leaves the question of the mean value of the pressure open.
  This additional data can be fixed,
  but it will have an effect on the time regularity of the pressure.
  For instance,
  if the pressure term $\pi_i$ is wanted to have mean value $m_i(t)$,
  we can take a pressure term from above and define a corrected pressure term 
  \[
  \tilde{\pi}_i = \pi_i - \la{\pi_i(t, \cdot), 1}/|\Omega_t| + m_i(t)/|\Omega_t|
  \]
  so that $\la{\tilde{\pi}_i(t,\cdot),1} = m(t)$.
  As the modification terms are space independent,
  the new pressure term $\tilde{\pi}_i$ satisfies the equations,
  but its regularity in time variable is affected by the desired mean value $m_i(t)$ as well as the volume of the domain,
  both of which can be non-trivial contributions. 
\end{remark}

The connection to the Lions--Ne\v{c}as theorem is the content of 
the following Proposition.

\begin{proposition}
  %\label{prop:negnorthe}
  Assume that $\Omega$ is $C^{\alpha,\beta,\theta}$ with $\partial_t | \Omega_t | = 0$.
  Then for all $k, \kappa \ge 0$ and $1<p,q<\infty$, 
  we have that for all $f \in C_0^{\infty}(\Omega)$ with mean value zero at every time slice
  \begin{multline*}
  \no{\nabla f}_{L^{q}W_0^{-k,p,-\beta,-\kappa,-\alpha\beta,0}(\Omega)}
  \le C \no{f}_{L^{q}W_0^{-k+1,p,-\beta,-\kappa,-\alpha\beta,0}(\Omega)} \\
  \le C \no{\nabla f \circ \partial_t^{-\kappa}}_{L^{q}W^{-k,p,\infty,0,\infty,1-\beta}(\Omega)}.
  \end{multline*}
\end{proposition}
\begin{proof} 
  Consider a function $f \in C_0^{\infty}(\Omega)$
  with  
  \[
  \int_{\Omega_t} f(t,x) \, dx = 0
  \] 
  for all $t$.
  For test functions $\psi \in C^{\infty}(\Omega_t;\rn)$
\begin{multline*}  
  | \la{\nabla f, \psi } | 
  = | \la{f, \dive \psi} | \\
  \le  \no{f(t, \cdot) }_{L^{q}W_0^{-k,p,-\beta,-\kappa,-\alpha\beta,0}(\Omega)} \no{ \dive \psi(t, \cdot)}_{L^{q'}W^{k,p',\beta,\kappa,\alpha\beta,0}(\Omega)} .
\end{multline*}
  
To prove the reverse inequality,
we define $\pi$ on test functions with mean value zero 
through $\la{\pi, \psi} = \la{\nabla f, B \psi}$.
Then $\pi(t,\cdot) = f(t,\cdot) + c(t) $
and we may choose $\pi(t,x)$, 
now on all test functions, 
to be smooth and with mean value zero so that $c(t) = 0$.
The second item of Lemma \ref{lem:distribution} still applies.
We see that 
\[
  \no{f}_{L^{q}W_0^{-k+1,p,-\beta,-\kappa,\alpha\beta,0}(\Omega)}
  \le C \no{\nabla f \circ \partial_t^{-\kappa}}_{L^{q}W^{-k,p,\infty,0,\infty,\beta -1}(\Omega)}
\]
This finishes the proof.
\end{proof}
 
\subsection{Pressure estimates} 
\label{sec:fluidthms}
We turn to two concrete applications. 
The first one is on very weak solutions of the Dirichlet problem. 
The second one is about weak solutions with a slip boundary condition.
As was mentioned in the introduction, 
the existence of a velocity field $v$ is known in many cases \cite{neustupa09,NeuPan09}, 
but the pressure is commonly introduced only
as an abstract Lagrange multiplier, 
and it is not even included in the respective weak formulation in many cases. 

In accordance with the focus of the current manuscript,
we consider a pressure 
that is global in space but local in time.
As we consider local-in-time solutions,
we will not mention initial values among the boundary conditions here.
We actually consider solutions that are defined on the open interval $(0,T)$. 
In view of the Bogovskij estimates, 
respective global in time pressure estimates for Cauchy problems 
can also be deduced without further difficulty. 

First let us discuss the important case of Dirichlet boundary conditions
\begin{align*}
v(t) &=v_0(t) , \quad \trm{in $\partial \Omega_t$,}
\end{align*}
for almost every $t\in (0,T)$. 
Dirichlet boundary conditions are commonly defined 
via the choice of the function space in which the solutions are to be found. 
Consequently, Dirichlet boundary conditions are only well-defined for functions 
that are smooth enough.
That is, 
the boundary of the domain must have non-zero capacity relative to the relevant function space.
In view of the reconstruction of a pressure,
the Dirichlet boundary values of the velocity are practically not seen at all.
In particular, a pressure may be reconstructed under assumptions much weaker 
than $v\in L^1(0,T;W^{1,1}(\Omega_t))$. 
Hence we assume $v \in L_{loc}^{2}(\Omega)$ in what follows,
and $v$ may or may not satisfy a Dirichlet boundary condition in whichever form.

Define 
\begin{multline}
  \label{eq:lambdaterms}
  \Lambda_1(\psi) = -\la{ v,\partial_t \psi}, \quad
  \Lambda_2(\psi) = - \la{v \otimes v ,\nabla \psi},\quad 
  \Lambda_3(\psi) = - \mu \la{ v , \Delta \psi },\\
  \Lambda_4(\psi) = - \la{g ,\psi}, \quad 
  \Lambda_5(\psi) =  \la{F ,\nabla \psi}
\end{multline}
where the forces $f$ and $F$ are the given right hand sides of the equations, 
possibly just distributions.
Restating Definition \ref{def:fluid_intro},
the pair $(v,\pi)$ is called a local very weak solution to \eqref{eq:NS},
possibly coming from a Dirichlet boundary data,
if 
\begin{equation}
  \label{eq:dbc1}
 \sum_{i=1}^{5} \Lambda_i(\psi) = 0 
\end{equation}
holds for all $ \psi \in C_{sol,0}^{\infty}(\Omega;\rn) $,
and
\begin{align}
\label{eq:dbc3}
\la{\pi, \dive \psi} = - \sum_{i=1}^{5} \Lambda_i(\psi)\text{ and }\la{v,\nabla\psi}=0
\end{align}
hold for all $ \psi \in C_{0}^{\infty}(\Omega;\rn) $.
We can directly apply Lemma \ref{lem:distribution} to conclude the following statement.
\begin{theorem}
  \label{thm:fluid1}
  Let $v \in L_{loc}^{b}(\Omega)$ for $b>2$ satisfy \eqref{eq:dbc1}.
  Then there exist $\pi_i$, $i=1,\ldots, 5$,
  such that $(v, \sum_i \pi_i)$ is a very weak solution to \eqref{eq:dbc3}
  and the estimates from Lemma \ref{lem:distribution} apply.
\end{theorem}

Next we consider general impermeable boundaries. 
For the motivation,
we assume that the domain and the velocity of the fluid are given and smooth. 
Actual results are stated in an abstract framework.
We define $\nu(t,\cdot)$ as the spatial outer normal of $\Omega$ 
and $\tau^j$ as the $j$th tangential vector, $j=1,...,n-1$.
In the case of cylindrical domains, 
the impermeability condition reads (locally)
\[
v\cdot \nu=0\trm{ in $\partial \Omega_0$}.
\]
In the case of non-cylindrical domains, 
it means that the fluid moves to the normal direction as fast as the boundary of the domain.
We can use the local graph coordinates and assume that locally 
$\Omega_t\cap R=\{(x',x_n)\in R:0<x_n<\psi(t,x')\}$. 
Given a boundary point $x=(x',\psi(t,x'))$,
the impermeability condition then reads 
\begin{align}
\label{eq:sbc1}
v(t,x)\cdot \nu(t,x)=\partial_t\psi(t,x')\nu_n(t,x)=-\frac{\partial_t\psi(t,x')}{\sqrt{1+\abs{\nabla_x\psi(t,x')}}^2}.
\end{align}
The motion of the fluid in the tangential direction 
is neither affected by the impermeability condition, 
nor by
the stress forces depending on the pressure at the boundary. 
Hence we may allow the flow to change tangentially either freely or 
according to an external force acting on the fluid-stresses in tangential direction:
\begin{align}
\label{eq:sbc2}
(\nabla_{\mathrm{sym}}v\tau^j) \cdot\nu &=s_0^j , \quad \trm{in $\partial \Omega_t$},
\end{align}
where we generally assume $s_0\in C_0^{\infty}(\R^{n+1})^{*}$ 
with $s_0(\phi)=0$ for all $\phi\in C^\infty_c(\Omega)$.
The respective force $s_0$ would typically depend on the fluid-velocity itself, 
the force $F$ or the motion of the domain 
or any combination of the previously listed \cite{NeuPan09}.

Next we give a weak formulation.
We ask
\begin{align*}
  -\la{ v,\partial_t \psi} - \la{v \otimes v ,\nabla \psi}+ \mu \la{\nabla_{\mathrm{sym}}  v , \nabla \psi }
  -  \la{g ,\psi} + \la{F ,\nabla \psi} &= \sum_{j} \la{s_0^j,\psi}  \\
  \la{v, \nabla \psi} &= 0
\end{align*}
to hold for $\psi\in \Tcal_{\nu,sol}$, 
where
\begin{align*}
  \Tcal_\nu&:=\bigg\{\psi\in  C^\infty(\R^{n+1},\R^n):\int_{\Omega}\psi\cdot \nabla\phi\,dx\, dt 
  \\
  &\qquad=-\int_\Omega (\dive\psi) \phi\, d x d t\text{ for all }\phi\in C^\infty(\R^{n+1})\bigg\}.
\\
\Tcal_{\nu,sol}&:=\{\psi\in \Tcal_\nu:\dive\psi=0\}.
\end{align*}
We define $\Lambda_i' = \Lambda_i $ for $i \in \{1,2,4,5\}$ as in \eqref{eq:lambdaterms}
and we set 
\[\la{\Lambda_3',\psi} = - \mu \la{\nabla_{\textrm{sym}}v , \nabla \psi }.\]
The pair $(v,\pi)$ is a local solution to \eqref{eq:NS} 
with boundary conditions \eqref{eq:sbc1} and \eqref{eq:sbc2}
if 
\begin{equation}
  \label{eq:wnod}
 \sum_{i=1}^{5} \Lambda_i'(\psi) = \sum_{j} \la{s_0^j,\psi}%+ N(\psi) 
\end{equation}
holds for all $ \psi \in \Tcal_{\nu,sol}$
and
\[
\la{\pi, \dive \varphi} = - \sum_{i=1}^{5} \Lambda_i(\varphi) + \sum_{j} \la{s_0^j,\varphi}%+N 
\]
holds for all $ \varphi \in \Tcal_{\nu} $.
We can directly apply Lemma \ref{lem:distribution}
with $A = \Tcal_{\nu}$
to conclude the following statement.
\begin{theorem}
  Let $a,b > 1$ and $c > 2$.
  Let $v \in  L^{c}(\Omega;\R^n)$
  with $1_{\Omega} |\nabla v| \in L^{a}L^{b}(\Omega)$
  satisfy \eqref{eq:wnod}.
  Then there exist $\pi_i$, $i=1,\ldots, 5$,
  such that $(v, \sum_i \pi_i)$ is a solution
  and the estimates from Lemma \ref{lem:distribution} apply.
\end{theorem}

We conclude the discussion by connecting these statements to the theorems in the introduction.
\begin{proof}[Proof of Theorem \ref{thmint:fluid2}]
We apply (3) of Lemma \ref{lem:distribution} to $\psi \mapsto \la{v, \partial_t \psi}$.
We have $\kappa = 1$ and 
we can use the indices $(0,p,\infty,0,\infty,1-\beta)$ on the right hand side.
Hence we conclude the norm of $\pi_{time}$ has its 
\[
W^{1,p,\infty,-1,-\alpha\beta,0}(t,\Omega)  
\]
norm bounded
by the relevant weighted $L^{p}$ norm of $v$.

To deal with the convection term $\psi \mapsto \la{v, \nabla \psi}$,
we use (2) with right hand side indices $(-1,p/2,\infty,0,\infty,1-\beta)$
to conclude the membership in 
\[
W^{0,p/2,-\beta,0,-\alpha\beta,0}(t,\Omega) = W^{0,p/2,\infty,0,\infty,0}(t,\Omega) = L^{p/2}(\Omega_t).
\]
To handle the second term on external force $\psi \mapsto \la{F, \nabla \psi}$,
we use (2) with right hand side indices $(-1,r,\infty,0,\infty,1 - \beta)$
giving 
\[
\pi_{ext,2} \in L^{s}W^{0,r,\infty,0,\infty,0}.
\]
To handle the viscosity term $\psi \mapsto \la{v, \Delta \psi}$,
we use the parameters  
\[
(-2,p,\infty,0,\infty,1-\beta)
\]  on the right hand side of (2).
We conclude the membership with boundedness as with $(-1,p,-\beta,0,\infty,0)$.
Finally, to deal with the first term of external force $\psi \mapsto \la{g, \psi}$,
we use (3) with right hand side indices $(0,r,\infty,0,\infty,1-\beta)$
to place the respective pressure term in 
\[
L^{s}W^{1,r,\infty,0,\infty,0}(\Omega).
\]
This justifies all the bounds claimed in the statement of the theorem.
\end{proof}

\begin{proof}[Proof of Theorem~\ref{thmint:fluid3}]
We start by deducing weighted bounds for the data.
Applying Proposition \ref{prop:tounweighted}
to functions $d^{\beta-1}|v|$,
$d^{\beta-1}|g|$
and $d^{\beta-1}|F|$
with ${{{{\eta}}}} = 1- \beta$,
we obtain
\begin{align*}
  &\int \left(\int |v(t,x)|^{\check{p}} \dist (x, \partial \Omega_t) ^{(\beta - 1)\check{p}} \, dx \right)^{q/\check{p}} \, dt < \infty, \\
  & \int \left( \int (|F(t,x)|^{ \check{r}}  + |g(t,x)|^{ \check{r}})\dist (x, \partial \Omega_t) ^{(\beta - 1) \check{r}} \, dx \right)^{s/ \check{r}}  dt < \infty.
\end{align*}
whenever 
\[
\frac{1-\beta}{\theta} < \min\left( \frac{1}{\check{r}} - \frac{1}{r}, \frac{1}{\check{p}} - \frac{1}{p}\right). 
\]
Hence we immediately obtain bounds for all terms except for $\psi \mapsto \la{v, \partial_t \psi}$
from Theorem \ref{thmint:fluid3}, 
provided we have all the integrability indices in the duality range $(1,\infty)$.

Consider then the term $\la{v, \partial_t \psi}$.
We use the decomposition  
from (4) of Lemma \ref{lem:distribution} 
to see that a pressure term with components
\[
\pi_{time}^{1} \in W^{-1,q}W^{1,p_1,\infty,0,\infty,0}(\Omega), \quad
  \pi_{time}^{2} \in L^{q}W^{1,p_2,\infty,0,\infty,0}(\Omega)
\]
can be constructed
provided that the condition 
\[
v \in L^{q}W^{0, p_1 ,\infty,0,\infty,-\beta + 1/\alpha}(\Omega) \cap  L^{q}W^{0,p_2,\infty,0,\infty,1-\beta}(\Omega)
\]
holds.
Here we have replaced one by infinity in the third index as the order of smoothness is zero.
Applying Proposition \ref{prop:tounweighted}
to $ d^{- 1/ \alpha + \beta }|v| $ and $d^{\beta -1}|v|$
with $q$ equal to $p_2$ and $p_1$,
${{{{\eta}}}} $ equal to $1/\alpha - \beta$ and $1- \beta$, respectively,
we see that the condition above is valid under the assumptions of Theorem \ref{thmint:fluid3}. 
\end{proof}

\begin{proof}[Proof of Theorem~\ref{thmint:fluid4}]
In addition to inserting $\beta = \theta = 1$ in all the estimates
in Theorem \ref{thmint:fluid3}
we have to show that 
\[
W_0^{-1,p,-1,0,\infty,0}(t,\Omega) \subset W^{-1,p,\infty,0,\infty,0}(t,\Omega)
\]
Obviously
\[
W_0^{-1,p,-1,0,\infty,0}(t,\Omega) \subset W^{-1,p,-1,0,\infty,0}(t,\Omega)
\]
as the functionals in the space on the left hand side act on $C^{\infty}$
whereas the functionals on the right hand side only act on $C_0^{\infty}$.
However, for $f \in C_{0}^{\infty}(\Omega)$ it follows from Hardy's inequality,
valid in bounded Lipschitz domains,
that 
\[
\no{f}_{W^{1,p',1,0,\infty,0}(t,\Omega)} \sim \no{f}_{W^{1,p'}(t,\Omega)}
\]
and hence the claim follows.
\end{proof}
\subsection*{Data availability statement}
Data sharing not applicable to this article as no datasets were generated or analysed during the current study.

\subsection*{Conflict of interest statement}
On behalf of all authors, the corresponding author states that there is no conflict of interest.

\bibliography{references}
\bibliographystyle{abbrv}
%\printbibliography

\end{document}